\documentclass{article}
\usepackage[style=numeric,natbib=true,backend=bibtex,firstinits=true,maxnames=10]{biblatex}
\bibliography{biblio}
\input{minilib}
\input{subsup}
\newcommand{\defmeso}[2]{\newsubsupcommand{#1}{#2}[][\varepsilon]}
\newcommand{\defmesoi}[2]{\newsubsupcommand{#1}{#2}[i][\varepsilon]}
\defineDomain{Time}{t}{(0,T)}
\renewInt{Time}{\int\limits_0^T}
\defineDomain{A}{x}{\Omega^\varepsilon}
\defineBoundarySegment{G}{A}{Grain}{\Gamma^\varepsilon}
\defineBoundarySegment{ANeumann}{A}{Neumann}{\Gamma_N^\varepsilon}
\defineBoundarySegment{ARobin}{A}{Robin}{\Gamma_R^\varepsilon}
\defineBoundarySegment{ATRobin}{A}{TRobin}{\Sigma_R}
\defineBoundarySegment{ATDirichlet}{A}{TDirichlet}{\Sigma_D}
\defineCrossProduct{TA}{Time}{A}
\defineCrossProduct{TG}{Time}{G}

\newcommand\tempBase{\theta}
\newcommand\collBase{u}
\newcommand\depBase{v}
\defmeso{\temper}{\tempBase}
\defmeso{\coller}{\collBase}
\defmeso{\deper}{\depBase}
\defineUnknown{temp}{\temper}{TA}
\defineUnknown{coll}{\coller}{TA}
\defineUnknown{dep}{\deper}{TG}
\inheritUnknown{coll}{coll_i}{\coller_i}
\inheritUnknown{dep}{dep_i}{\deper_i}
\defmesoi{\colli}{u}
\defmesoi{\depi}{v}

\defmeso{\cond}{\kappa}
\defmeso{\diff}{d}
\defmeso{\soret}{\tau}
\defmeso{\dufour}{\delta}
\attachFlux{temp}{\cond}
\attachFlux{coll}{\diff}
\attachFlux{coll_i}{\diff_i}
\newcommand{\bcoll}{\bar{\collBase}}
\newcommand{\btemp}{\bar{\tempBase}}

\newcommand{\coag}{\beta}
\newcommand{\cflux}{g}
\defmeso{\dca}{a}
\defmeso{\dcb}{b}
\newcommand{\loa}{{L^1(\A)}}
\newcommand{\lta}{{L^2(\A)}}
\newcommand{\lfa}{{L^4(\A)}}
\newcommand{\lia}{{L^\infty(\A)}}
\newcommand{\liap}{{L^\infty_+(\A)}}

\newcommand{\ha}{{H^1(\A)}}
\newcommand{\ltha}{{L^2(0,T;H^1(\A))}}
\newcommand{\ltlta}{{L^2(0,T;L^2(\A))}}

\newcommand{\lilia}{{L^\infty((0,T)\times\A)}}

\newcommand{\lilta}{{L^\infty(0,T;L^2(\A))}}
\newcommand{\liha}{{L^\infty(0,T;H^1(\A))}}
\newcommand{\hlta}{{H^1(0,T;L^2(\A))}}

\newcommand{\ltr}{{L^2(\A[Robin])}}
\newcommand{\ltg}{{L^2(\A[Grain])}}

\newcommand{\ligp}{{L^\infty_+(\G)}}
\newcommand{\lilig}{{L^\infty((0,T)\times\G)}}

\newcommand{\liltg}{{L^\infty(0,T;L^2(\G))}}
\newcommand{\hltg}{{H^1(0,T;L^2(\G))}}

\newcommand\weakTo{\rightharpoonup}
\newcommand\twoScaleTo{\stackrel{2}{\rightharpoonup}}
\newcommand\weakStarTo{\stackrel{*}{\rightharpoonup}}
\newcommand{\inta}{\int\limits_{\A}}

\newcommand{\intr}{\int\limits_{\A[Robin]}}
\newcommand{\intg}{\int\limits_{\A[Grain]}}
\newcommand{\intt}{\int\limits_0^T}
\newcommand{\sumN}{\sum_{i=1}^N}
\newcommand{\half}{\frac{1}{2}}
\newcommand{\naturalNumbers}{\mathbb{N}}
\newcommand\mspan{\mathrm{span}}
\newcommand{\Rdim}[1]{\mathbb{R}^#1}
\newcommand{\Zdim}[1]{\mathbb{Z}^#1}
\newcommand{\Runit}[1]{\vec{e}_#1}
\newcommand{\operatorT}{\mathbf{T}}
\newcommandx\newcommandWithIndex[3]{
  \constructor{#1}{#2#3}
  \set{#1}{body}{#2}
  \set{#1}{index}{#3}
}
\makeatletter
\newcommand{\closure}[1]{
  \@ifundefined{#1@body}
  {\overline{#1}}
  {\overline{\@nameuse{#1@body}}\@nameuse{#1@index}}
}
\makeatother
\defineDomain{Domain}{x}{\Omega}
\defineBoundarySegment{DTNeumann}{Domain}{TNeumann}{\Gamma_N^\theta}
\defineBoundarySegment{DTRobin}{Domain}{TRobin}{\Gamma_R^\theta}
\defineBoundarySegment{DTDirichlet}{Domain}{TDirichlet}{\Gamma_D^\theta}
\defineBoundarySegment{DUNeumann}{Domain}{UNeumann}{\Gamma_N}
\defineBoundarySegment{DURobin}{Domain}{URobin}{\Gamma_R}
\defineBoundarySegment{DUDirichlet}{Domain}{UDirichlet}{\Gamma_D^u}
\newcommand{\cellUnit}{Y}
\newcommandWithIndex{cellFluid}{\cellUnit}{_1}
\newcommandWithIndex{cellSolid}{\cellUnit}{_0}
\newcommand{\cellSolidBoundary}{\Gamma}
\newcommandWithIndex{volSkeleton}{\Domain}{_0^\varepsilon}
\newcommandWithIndex{volPore}{\Domain}{^\varepsilon}
\newcommand{\surfSkeleton}{\Gamma^\varepsilon}
\newcommand{\model}{\nameref{epsmodel}}
\newcommand{\INPUT}[1]{#1}
\newcommandx{\inRange}[2][2=]{\ifstrequal{#2}{}
  {\in\{1,\ldots,#1\}}
  {\in\{#1,\ldots,#2\}}
}

\newcommand\Diff{\bar{D}}

\newcommand\intY[1]{\int_Y#1\,dy}
\newcommand\intG[1]{\int_{\Gamma}#1\,d\sigma(y)}

\newcommand\initialCondition[1]{
  \self{#1}(0,\get{#1}{domain}[d2][var])=\self{#1}^0(\get{#1}{domain}[d2][var]),&&\tdomain{#1}[2]
}
\usepackage{graphicx}
\usepackage{amsmath}
\usepackage{amsfonts}
\usepackage{amsthm}
\newtheorem{theorem}{Theorem}[section]

\newtheorem{definition}{Definition}
\newtheorem{denote}{Notation}

\newtheorem{lemma}[theorem]{Lemma}

\newtheorem{remark}[theorem]{Remark}

\usepackage{enumerate}
\usepackage{empheq}
\usepackage[bookmarks=false]{hyperref}
\makeatletter\def\namedlabel#1#2{\begingroup#2\def\@currentlabel{#2}\phantomsection\label{#1}\endgroup}\makeatother
\title{A thermo-diffusion system with Smoluchowski interactions}

\author{Oleh Krehel and Adrian Muntean and Toyohiko Aiki}

\begin{document}
\maketitle
\centerline{\scshape Oleh Krehel}
\medskip
{\footnotesize
  \centerline{Department of Mathematics and Computer Science}
  \centerline{CASA - Center for Analysis, Scientific computing and Engineering}
  \centerline{Eindhoven University of Technology}
  \centerline{5600 MB, PO Box 513, Eindhoven, The Netherlands}
} 
\medskip
\centerline{\scshape Toyohiko Aiki}
\medskip
{\footnotesize
  \centerline{Department of Mathematical and Physical Sciences, Faculty of Science}
  \centerline{Japan Women's University, Tokyo, Japan}
}
\medskip
\centerline{\scshape Adrian Muntean}
\medskip
{\footnotesize
  \centerline{Department of Mathematics and Computer Science}
  \centerline{CASA - Center for Analysis, Scientific computing and Engineering}
  \centerline{ICMS - Institute for Complex Molecular Systems}
  \centerline{Eindhoven University of Technology}
  \centerline{5600 MB, PO Box 513, Eindhoven The Netherlands}
}
\bigskip
\begin{abstract}
  We study the solvability and homogenization of a thermal-diffusion
  reaction problem posed in a periodically perforated domain. The system
  describes the motion of populations of hot colloidal particles
  interacting together via Smoluchowski production terms.  The upscaled
  system, obtained via two-scale convergence techniques, allows the
  investigation of deposition effects in porous materials in the
  presence of thermal gradients.
\end{abstract}

\section{Introduction}
We aim at understanding processes driven by coupled fluxes through
media with microstructures. In this paper, we study a particular
type of coupling: we look at the interplay between diffusion fluxes of
a fixed number of colloidal populations and a heat flux, the effects
included here are incorporating an approximation of the Dufour ad Soret
effects (cf. Section \ref{SD}, see also \cite{groot1962non}. The type
of system of evolution equations that we encounter in Section
\ref{eqs} resembles very much cross-diffusion and chemotaxis-like
systems; see e.g. \cite{vanag2009cross,funaki2012link}. The structure
of the chosen equations is useful in investigating transport,
interaction, and deposition of a large numbers of hot multiple-sized
particles in porous media.

Practical applications of our approach would include predicting the
response of refractory concrete to high-temperatures exposure in steel
furnaces, propagation of combustion waves due to explosions in
tunnels, drug delivery in biological tissues, etc.; see
for instance \cite{benevs2013global,benevs2013analysis,rothstein2009unified,soares2010mixture,gong1996development,golestanian2012collective}.
\ifdefined\included\else
In the paper \cite{krehel2014multiscale} we study
quantitatively some of these effects, focusing on colloids deposition
under thermal gradients.
\fi
Within this framework, our focus lies exclusively on two distinct theoretical aspects:
\begin{enumerate}
\item[(i)] the mathematical understanding of the microscopic problem
  (i.e. the well-posedness of the starting system);
\item[(ii)] the averaging of the thermo-diffusion system over arrays
  of periodically-distributed microstructures (the so-called, {\em
    homogenization asymptotics limit}; see, for instance,
  \cite{bensoussan2011asymptotic,mei2010homogenization} and references cited therein).
\end{enumerate}
The complexity of the microscopic system makes numerical simulations
on the macro scale very expensive. That is the reason that the aspect
(ii) is of concern here.  Obviously, the study does not close with
these questions. Many other issues like derivation of corrector
estimates, design of efficient convergent numerical multiscale schemes,
multiscale parameter identification etc. need also to be
treated. Possible generalizations could point out to coupling heat
transfer with Nernst-Planck-Stokes systems (extending
\cite{ray2012rigorous}) or with semiconductor equations
\cite{masmoudi2007diffusion}.
The paper is structured in the following manner. We present the basic
notation and explain the multiscale geometry as well as some of the
relevant physical processes in Section \ref{NA}.
Section \ref{solvability} contains the proof of the solvability of the
microstructure model. Finally, the homogenization procedure is
performed in Section \ref{homogenization}.
The strong formulation of the upscaled thermo-diffusion model with
Smoluchowski interactions is emphasized in Section \ref{strong}.

\section{Notations and assumptions}\label{NA}

\subsection{Model description and geometry}\label{geometry}

The geometry of the problem is depicted in Figure~\ref{fig:thermal.porous}.
\begin{table}[h!]
  \begin{tabular}{ll}
    $\Time$        & $=$ time interval of interest\\
    $\Domain$      & $=$  {$(0,L) \times \cdots  \times (0,L)$} bounded domain in $\Rdim{n}$
                     {for $L > 0$} \\
    { $\varepsilon$  }& $=$  {$\frac{L}{\ell}$ for any integer $\ell$} \\
    $\partial \Domain$     & $=$ piecewise smooth boundary of $\Domain$\\
    $\Runit{i}$    & $=$ $i$th unit vector in $\Rdim{n}$\\
    $\cellUnit$    & $=\{\sum_{i=1}^n\lambda_i\Runit{i}:\: 0<\lambda_i<1\}$ unit cell in $\Rdim{n}$\\
    $\cellSolid$   & $=$ open subset of $\cellUnit$ that represents the solid grain\\
    $\cellFluid$   & $=\cellUnit \setminus  \closure{cellSolid}$ \\
    $\cellSolidBoundary$    & $=\partial \cellSolid$ piecewise smooth boundary of $\cellSolid$\\
    $X^k$          & $=X+\sum_{i=1}^nk_i\Runit{i}$, where $k=(k_1,\ldots,k_n)\in \Zdim{n}$ and $X\subset\cellUnit$\\
    $\volSkeleton$ & $=\cup\{(\varepsilon \cellSolid)^k:\: (\cellSolid)^k\subset \A, k\in \Zdim{n}\}$ pore skeleton\\
    $\volPore$     & $=\Domain\setminus\closure{volSkeleton}$ pore space\\
    $\surfSkeleton$ & $=\partial \volSkeleton$ boundary of the pore skeleton
  \end{tabular}
\end{table}
The standard cell is shown in Figure~\ref{fig:model-geometry}.

The cells regions without the grain $\varepsilon \cellFluid^{k}$ are filled with water
and we denote their union by $\volPore$.
Colloidal species are dissolved in the pore water. They react between themselves
and participate in diffusion and convective transport.
The colloidal matter cannot penetrate the grain boundary $\A[Grain]$,
but it deposits there reducing the amount of mass floating inside $\A$.
Here $\partial \A=\partial \Domain\cup \A[Grain]$, where $\A[Grain]=\A[Neumann]\cup \A[Robin]$ and $\A[Neumann]\cap \A[Robin]=\emptyset$.
The boundary $\A[Neumann]$ is insulated to the heat flow, while $\A[Robin]$ admits flux.

\begin{figure}[h!]
  \includegraphics[width=200pt]{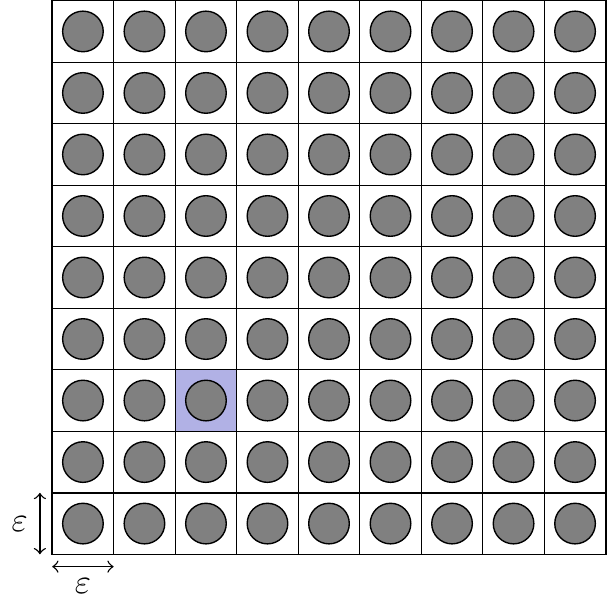}
  \caption{Porous medium geometry {
      $\Omega^{\varepsilon} = \Omega \backslash \Omega_0^{\varepsilon}$,}
    where the pore skeleton $\volSkeleton$ is marked with gray color and the pore space $\Omega^{\varepsilon}$ is white. }
  \label{fig:thermal.porous}
\end{figure}
\begin{figure}[h!]
  \begin{center}
    \includegraphics[width=200pt]{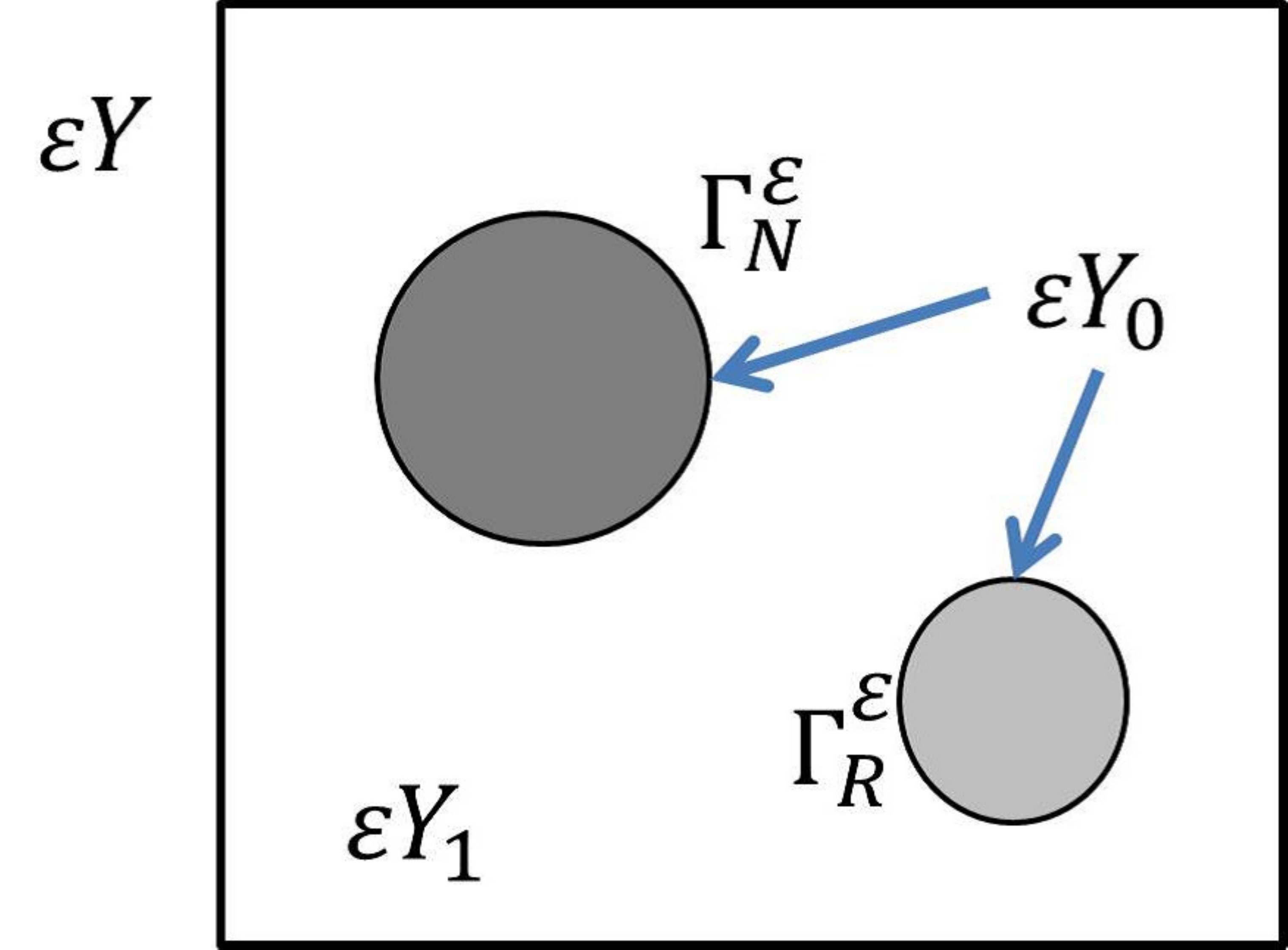}
  \end{center}
  \caption{ The unit cell geometry. The colloidal species $\coll_i$ and
    temperature $\temp$ are defined in { $\Omega^{\varepsilon}$ , while the
      deposited species $\dep_i$ are defined on $\Gamma^{\varepsilon} =  \Gamma_R^{\varepsilon}  \cup  \Gamma_N^{\varepsilon}$.
      The boundary conditions for $\temp$ differ on
      $\Domain[URobin]$ and $\Domain[UNeumann]$, while the boundary
      conditions for $\coll_i$ are uniform on $\Gamma^{\varepsilon}$. }  }
  \label{fig:model-geometry}
\end{figure}
\newpage
The unknowns are:
\begin{itemize}
\item $\temp$ -- the temperature in $\A$.
\item $\coll_i$ -- the concentration of the species that contains $i$ monomers in $\A$.
\item $\dep_i$ -- the mass of the deposited species on $\G$.
\end{itemize}

Furthermore, for a given $\delta>0$ we introduce the mollifier:
\begin{align}
  \label{eq:thermal.mollifier}
  &J_\delta(s):=
    \begin{cases}
      Ce^{1/(|s|^2-\delta^2)} & \text{if } |s| < \delta,\\
      0                   & \text{if } |s| \ge  \delta,
    \end{cases}
                            \intertext{where the constant $C>0$ is selected such that}
                            \nonumber
                              &\int_{\Rdim{d}}J_\delta=1,
\end{align}
see \cite{evans1998partial} for details.

Using $J_\delta$ from (\ref{eq:thermal.mollifier}), define the mollified gradient:
\begin{equation}
  \label{eq:thermal.mollified-gradient}
  \nabla^\delta f:=\nabla
  \left[
    \int_{B(x,\delta)}J_\delta(x-y)f(y)dy
  \right].
\end{equation}
The following statement holds for all { $1\le p\le \infty$:
  \begin{align}
    \label{eq:thermal.mollified-gradient-property}
    &\|\nabla ^\delta f \cdot  g\|_{\Lp{A}} \le  c^\delta \|f\|_{\Li{A}} \|g\|_{\Lp{A}^d}
      \mbox{ for } f \in L^{\infty}(\Omega^{\varepsilon}), g \in L^{p}(\Omega^{\varepsilon})^d,
    \\
    &\|\nabla ^\delta f\|_{\Lp{A}} \le  c^\delta\|f\|_\lta
      \mbox{ for } f \in L^{2}(\Omega^{\varepsilon}).
  \end{align}
}
In the equations below all norms are $\Lt{A}$ unless specified otherwise,
with $c^\delta$ independent of the choice of $\varepsilon$.

\subsection{Smoluchowski population balance equations}\label{S}

We want to model the transport of aggregating colloidal particles
under the influence of thermal gradients.  We use the Smoluchowski
population balance equation, originally proposed in
\cite{smoluchowski1917versuch}, to account for colloidal aggregation:
\begin{align}
  \label{eq:aggregation-Smoluchowski}
  R_i(s):=\half\sum_{k+j=i}\coag_{kj}s_ks_j
  -\sum_{j=1}^N\coag_{ij}s_is_j, &&i\inRange{N};\:N>2.
\end{align}
\ifdefined\included
Equation (\ref{eq:aggregation-Smoluchowski}) is explained in detail in Chapter \ref{ch:modeling}.
\fi
Here $s_i$ is the concentration of the colloidal species that consists of $i$ monomers,
$N$ is the number of species, i.e. the maximal aggregate size that we consider,
$R_i(s)$ is the rate of change of $s_i$,
and $\coag_{ij}>0$ are the coagulation coefficients, which tell us the rate aggregation
between particles of size $i$ and $j$ \cite{elimelech1998particle}.
Colloidal aggregation rates are described in more detail in \cite{krehel2012flocculation}.


\subsection{Soret and Dufour effects}\label{SD}
\index{Soret effect}
\index{Dufour effect}
The system we have in mind is inspired by the model proposed by Shigesada, Kawasaki
and Teramoto \cite{shigesada1979spatial} in 1979 when they have studied
the segregation of competing species.  For the case of two interacting species $u$
and $v$, the diffusion term looks like:
\begin{equation}
  \label{eq:cross-diffusion-Shigesada}
  \partial _t u=\Delta(d_1u+\alpha uv),
\end{equation}
where the second term in the flux is due to cross-diffusion. The second term can
be expressed as:
\begin{equation}
  \label{eq:laplacion-of-product}
  \Delta(uv)=u\Delta v+v\Delta u+2\nabla u\cdot\nabla v.
\end{equation}
As a first step in our approach, we consider only the last term of
(\ref{eq:laplacion-of-product}), i.e. $\nabla u\cdot \nabla v$, as the driving force of
cross-diffusion and we postpone the study of terms $u\Delta v$ and $v\Delta u$
until later.

From mathematical point of view, still it is not easy to treat the term $\nabla u\cdot \nabla v$.
Hence, in the paper we approximate this term by $\nabla ^\delta u\cdot \nabla v$ for $\delta>0$.


\subsection{Setting of the model equations} \label{eqs}

\newcommand\eqTemperature[2]{
  \partial _t\self{#1} + \nabla \cdot (-\flux{#1}) - \soret\sumN \nabla ^\delta#2_i\cdot \nabla \self{#1} = 0
}
\newcommandx\alignEq[4][4=]{
  #1{#2}{#3}{#4}, && \tdomain{#2}
}
\newcommand\eqColloids[3]{
  \partial _t\self{#1} + \nabla \cdot (-\flux{#1}) - \dufour_i \nabla ^\delta#2\cdot \nabla \self{#1} = #3
}
\newcommand\eqDeposit[2]{
  \partial _t\self{#1}= a_i#2- b_i\self{#1}
}
We consider the following balance equations for the temperature and colloid concentrations:
\begin{description}
\item[\namedlabel{epsmodel}{\bf ($P^\varepsilon$)}]
  \begin{align}
    \label{eq:main-temp}
    &\alignEq{\eqTemperature}{temp}{\coll},\\
    \label{eq:main-coll}
    &\alignEq{\eqColloids}{coll_i}{\temp}[R_i(\coll)],\\
    \intertext{with boundary conditions:}
    \label{eq:bc-temp-neumann}
    & -\flux{temp}\cdot {\nu} =0,                                  &&\tdomain{temp}[Neumann],\\
    \label{eq:bc-temp-robin}
    & -\flux{temp}\cdot {\nu} =\varepsilon \cflux_0\temp,                        &&\tdomain{temp}[Robin],\\
    \label{eq:bc-temp-outer}
    & -\flux{temp}\cdot {\nu} =0,                             &&\text{on }\partial \Domain,\\
    \label{eq:bc-coll-outer}
    & -\flux{coll_i}\cdot {\nu} =0,                           &&\text{on }\partial \Domain,\\
    \intertext{ {where ${\nu}$ is the outward normal vector on the boundary} and a boundary condition for colloidal deposition:}
    \label{eq:bc-coll-grain}
    & -\flux{coll_i}\cdot {\nu} =\varepsilon (a_i\coll_i- b_i\dep_i),        &&\tdomain{coll_i}[Grain],\\
    \label{eq:main-dep}
    & \partial_t v_i^{\varepsilon} = a_i u_i^{\varepsilon} - b_i v_i^{\varepsilon}  &&\tdomain{coll_i}[Grain].
    \\
    \intertext{As initial conditions, we take for $i\inRange{N}$:}
    \label{eq:ic-temp}   & \initialCondition{temp},\\
    \label{eq:ic-coll} & \initialCondition{coll_i},\\
    \label{eq:ic-dep}       & \initialCondition{dep_i}.
  \end{align}
\end{description}
\begin{table}[h!]
  \caption{Physical parameters of \nameref{epsmodel}.}
  \begin{tabular}{ll}
    $\cond$   & heat conduction coefficient\\
    $\diff_i$ & diffusion coefficient\\
    $\soret$  & Soret coefficient\\
    $\dufour$ & Dufour coefficient\\
    $\cflux_i$ & Robin boundary coefficient, $i\inRange{0}[N]$\\
    $a_i$ & Deposition coefficient 1, $i\inRange{N}$\\
    $b_i$ & Deposition coefficient 2, $i\inRange{N}$\\
  \end{tabular}
\end{table}

We refer to  (\ref{eq:main-temp})-~(\ref{eq:ic-dep}) as \nameref{epsmodel} -- our reference microscopic model.
Note that the Soret and Dufour coefficients determine the structure of the particular
cross-diffusion system (see \cite{groot1962non},
\cite{shigesada1979spatial} \cite{andreianov2011analysis},
\cite{benevs2013global}, \cite{ni1998diffusion},
\cite{vanag2009cross}).
The coefficients $a_i$ and $b_i$ describe the deposition interaction
between $\coll_i$ and $\dep_i$.
Consequently, each $\coll_i$ has a different affinity to sediment as well as a different mass.

All functions defined in $\volPore$  are taken to be $\varepsilon$-periodic,
i.e. $\cond(x)=\kappa(x/\varepsilon)$ and so on.

Note the use of the mollified gradient in the cross diffusion terms in
(\ref{eq:main-temp}) and (\ref{eq:main-coll}). This is a choice that
we have to make at this point in order to obtain the necessary
estimates for our equations. From a physical point of view, smoothed
gradients causing advection can be interpreted as there being no turbulence.

\subsection{Assumptions on data}

\begin{description}
\item[\namedlabel{assumption:1}{\bf ($A_1$)}]
  {
    $\kappa$, $\tau$, $d_i$, $\delta_i \in L^{\infty}(Y)$ for each  $i\inRange{N}$.
    Moreover, $\kappa_0 \leq \kappa \leq \kappa_*$, $\tau \leq \tau_*$, $d_0 \leq d_i \leq d_*$,
    $\delta_i \leq \delta_*$ on $Y$  for $i\inRange{N}$, where $\kappa_0, \kappa_*, d_0,
    d_*$ and $\delta_*$  are positive constants. Also,
    $a_i$ and $b_i$ are positive constants for   $i\inRange{N}$, and we put
    $a_0 = \min(a_1, a_2, \ldots, a_N)$, $a_* = \max(a_1, a_2, \ldots, a_N)$, and
    $b_* = \max(b_1, b_2, \ldots, b_N)$.
  }

\item[\namedlabel{assumption:2}{\bf ($A_2$)}]
  $\temp^0\in \liap\cap { H^1(\Omega^\varepsilon)},$
  $\coll_i^0\in \liap\cap { H^1(\Omega^\varepsilon)},$
  $\dep_i^0\in\ligp$ for $i\inRange{N}$ and $\varepsilon >0$.
  Moreover, { $|| \temp^0||_{H^1(\Omega^\varepsilon)}\leq C_0$},
  {$||\coll_i^0||_{H^1(\Omega^\varepsilon)}\leq C_0$}, and
  $||\dep_i^0||_{L^\infty(\Gamma^\varepsilon)} \leq C_0$ for $i\inRange{N}$
  and $\varepsilon >0$. Here $C_0$ is a positive constant independent of $\varepsilon $.
  {Also, $ \liap = \{z \in L^{\infty}(\Omega^{\varepsilon}): z \geq 0 \mbox{ a.e. on } \Omega^{\varepsilon} \}$ and
    $ \ligp = \{z \in L^{\infty}(\Gamma^{\varepsilon}): z \geq 0 \mbox{ a.e. on } \Gamma^{\varepsilon} \}$.   }
\end{description}
\newcommand\allAssumptions{\nameref{assumption:1}-\nameref{assumption:2}}



\section{Global solvability of problem $(P^\varepsilon)$}\label{solvability}

\newcommand\eqTemperatureWeak[3]{
  \int[A] \partial _t#1#3 + \int[A] \cond\nabla #1\cdot \nabla #3 + {\varepsilon}   \cflux_0\intr#1#3
  =\sumN\inta \soret  \nabla^\delta#2\cdot \nabla #1#3
}
\newcommand\eqColloidsWeak[4]{
  &\int[A] \partial _t#1#4 + \int[A] \diff_i\nabla #1\cdot \nabla #4
  + { \varepsilon} \intg (a_i#1- b_i\dep_i)#4\\
  &\qquad=\int[A]  \dufour_i \nabla ^\delta #2\cdot \nabla #1#4
  +\int[A] #3#4
}
\newcommand\eqDepositWeak[3]{
  \intg \partial _t#1 #3= \intg (a_i#2- b_i#1)#3
}
\begin{definition}
  \label{def:weak-solution}
  The triplet $(\temp,\coll_i,\dep_i)$ is a solution to problem \model
  if the following holds:
  \begin{align}
    &\begin{aligned}
      &\temp,\coll_i\in \hlta \cap  \liha \cap  \lilia,\\
      &\dep_i\in \hltg \cap  \lilig,\\
    \end{aligned}
    \intertext{for all $\phi\in \ha:$} &\eqTemperatureWeak{\temp}{\coll_i}{\phi},\\
    \intertext{for all $\psi_i\in \ha:$}
    &\begin{aligned}
      \eqColloidsWeak{\coll_i}{\temp}{R_i(\coll)}{\psi_i},
    \end{aligned}\\
    \intertext{for all $\varphi_i\in \ltg:$}
    &\eqDepositWeak{\dep_i}{\coll_i}{\varphi_i},
  \end{align}
  together with (\ref{eq:ic-temp}), (\ref{eq:ic-coll}) and (\ref{eq:ic-dep})
  for a fixed value of $\varepsilon >0$.
\end{definition}
\begin{remark}
  We note that each term appearing in Definition \ref{def:weak-solution} is finite, since
  $\nabla ^\delta\coll_i$ and $\nabla ^\delta\temp$ are bounded in $\A$ due to (\ref{eq:thermal.mollified-gradient-property}).
\end{remark}
To prove the existence of solutions to problem \model, we introduce the following auxiliary problems
as iterations steps of the coupled system:
\begin{description}
\item[\namedlabel{auxiliary:1}{\bf ($P_1$)}]
  \begin{align*}
    &\alignEq{\eqTemperature}{temp}{\bcoll},\\
    & -\flux{temp}\cdot {\nu} =0,                                &&\tdomain{temp}[Neumann],\\
    & -\flux{temp}\cdot {\nu} =\varepsilon\cflux_0\temp,                      &&\tdomain{temp}[Robin],\\
    & -\flux{temp}\cdot {\nu} =0,                                  &&  { \mbox{on } (0,T) \times \partial \Omega},\\
    & \initialCondition{temp},
  \end{align*}
\end{description}
and
\begin{description}
\item[\namedlabel{auxiliary:2}{\bf ($P_2$)}]
  \begin{align*}
    & \alignEq{\eqColloids}{coll_i}{\btemp}[R_i^M(\coll)], \\
    & -\flux{coll_i}\cdot {\nu} =0,                             &&{ \mbox{on } (0,T) \times \partial \Omega}, \\
    & -\flux{coll_i}\cdot {\nu} =\varepsilon (a_i\coll_i- b_i\dep_i),       &&\tdomain{coll_i}[Grain],\\
    & \initialCondition{coll_i},\\
    & \alignEq{\eqDeposit}{dep_i}{\coll_i},\\
    & \initialCondition{dep_i}.
  \end{align*}
\end{description}
Here
\begin{align}
  \label{eq:coll-rhs-bounded}
  &R_i^M(s):=R_i(\sigma_M(s_1), \sigma_M(s_2), \ldots, \sigma_M(s_N)),\text{ for }s\in\Rdim{N}\\
  \intertext{denotes our choice of truncation of $R_i$, where}
  \label{eq:sigmam}
  &\sigma_M(r):=
    \begin{cases}
      0,&r<0, \\
      r,&r\in [0,M], \\
      M,&r>M,
    \end{cases}
\end{align}
where $M>0$ is a fixed threshold. Note that if $M$ is large enough, the essential bounds
obtained later in this paper will remain below $M$. This means that the existence result
is obtained also for the uncut rates.

In the following, assuming \nameref{assumption:1}-\nameref{assumption:2},
we show the existence, positivity and boundedness
of solutions to \nameref{auxiliary:1} and \nameref{auxiliary:2}.

When we denote the solutions of $P_1(\bcoll)$ by  $\temp$ and of $P_2(\btemp)$ by
$(\coll_i,\dep_i)$, respectively, we can define the solution operators
$(\temp,\coll_i) =\mathbf{T}(\btemp,\bcoll_i)$ and $\dep_i = \mathbf{T}_2(\btemp,\bcoll_i)$.
We will show that the operator $\mathbf{T}$ is a contraction in the appropriate functional spaces
and use the Banach fixed point theorem
to prove the existence and uniqueness of solutions to \model.
\begin{denote}
  Let $K(T,M):=\{z \in  \ltlta:\: |z| \le  M \text{ a.e. on }\TA\}$.
\end{denote}
\begin{lemma}{\bf Existence of solutions to \nameref{auxiliary:1}.}
  \label{lemma:P1-existence}\\
  \INPUT{
    Let $\bcoll_i \in  K(T,M)$,
    and assume that \nameref{assumption:1}-\nameref{assumption:2} hold.
  }  Then there exists $\temp \in  \hlta \cap  \liha$ that solves \nameref{auxiliary:1} in the sense:
  \begin{align}
    \intertext{for all $\phi\in \ha$ and a.e. in $[0,T]$:}
    \label{eq:P1-weak}
    &\eqTemperatureWeak{\temp}{\bcoll_i}{\phi},\\
    \intertext{and}
    \label{eq:P1-initial}
    &\temp(0,x)=\temp^0(x) \qquad \text{a.e. in $\A$.}
  \end{align}
\end{lemma}
\begin{proof}
  Let \{$\xi_i$\} be a Schauder basis of $\ha$.
  Then for each $n \in  \naturalNumbers$ there exists
  \begin{equation}
    \label{eq:P1-Galerkin-initial}
    \temp_n^0(x):=\sum_{j=1}^n\alpha_j^{0,n}\xi_j(x)
    \text{ such that }
    \temp_n^0\to\temp^0\text{ in }\ha
    \text{ as } n\to\infty.
  \end{equation}
  We denote by $\temp_n$ the Galerkin approximation of $\temp$, that is:
  \begin{align}
    \label{eq:P1-Galerkin-approximation}
    &\temp_n(t,x):=\sum_{j=1}^n\alpha_j^n(t)\xi_j(x) && \forAll{TA}.
  \end{align}
  By definition, $\temp_n$ must satisfy (\ref{eq:P1-weak}) for all $\phi\in \mspan\{\xi_j\}_{j=1}^n$, i.e.:
  \begin{align}
    \label{eq:P1-weak-finite}
    &\eqTemperatureWeak{\temp_n}{\bcoll_i}{\phi}.
  \end{align}
  The coefficients $\alpha_i^n(t)$ can be found
  by testing (\ref{eq:P1-weak-finite}) with $\phi:=\xi_i$
  and using (\ref{eq:P1-Galerkin-initial}) to solve the resulting ODE system:
  \begin{align}
    \label{eq:P1-ODE}
    &\partial _t\alpha_i^n(t)+\sum_{j=1}^n(A_{ij}+B_{ij}-C_{ij})\alpha_j^n(t)=0,&i\inRange{n},\\
    \label{eq:P1-ODE-initial}
    &\alpha_i^n(0)=\alpha_i^{0,n}.
  \end{align}
  The coefficients in (\ref{eq:P1-ODE}) and (\ref{eq:P1-ODE-initial})
  are defined by the following expressions
  \begin{align*}
    &A_{ij}:=\int[A] \cond \nabla \xi_i\cdot \nabla \xi_j,&&i,j\inRange{n},\\
    &B_{ij}:= { \varepsilon}  \cflux_0\intr \xi_i\xi_j,&&i,j\inRange{n},\\
    &C_{ij}:=\sum_{k=1}^N\inta \soret \nabla ^\delta\bcoll_k\cdot \nabla \xi_j \xi_i&&i,j\inRange{n}.
  \end{align*}
  Since the system (\ref{eq:P1-ODE}) is linear,
  there exists for each fixed $n\in \naturalNumbers$ a unique solution $\alpha_i^n \in  C^1([0,T])$.

  To prove uniform estimates for $\temp_n$ with respect to $n$,
  we take in (\ref{eq:P1-weak-finite}) $\phi=\temp_n$. We obtain:
  \begin{align*}
    \half\partial _t\| \temp_n\| ^2
    + {\kappa_0} \| \nabla \temp_n\| ^2
    + { \varepsilon} \cflux_0\| \temp_n\| ^2_\ltr
    \le
    \sumN\inta \soret  |\nabla ^\delta \bcoll_i \cdot  \nabla \temp_n\temp_n|
    :=
    \tau_* \sumN A_i.
  \end{align*}
  Using the Cauchy-Schwarz inequality and
  Young's inequality in the form\\$ab\le\eta a^2+b^2/4\eta $,
  where $\eta >0$, we get:
  \begin{equation*}
    A_i\le \eta \| \nabla \temp_n\| ^2+\frac{1}{4\eta }\| \nabla ^\delta \bcoll_i\temp_n\| ^2
    \le \eta \| \nabla \temp_n\| ^2+\frac{1}{4\eta }\| \nabla ^\delta \bcoll_i\| ^2_{\lfa}\| \temp_n\| ^2_{\lfa}.
  \end{equation*}
  The mollifier property (\ref{eq:thermal.mollified-gradient-property}) yields
  $\| \nabla ^\delta \bcoll_i\| ^2_{\lfa}\le c^\delta\| \bcoll_i\| ^2_\infty$.
  Using Gagliardo-Nirenberg inequality (see \cite{nirenberg1959elliptic} e.g.), we get:
  \begin{equation}
    \label{eq:P1-Nirenberg}
    \| \temp_n\| ^2_{\lfa}\le c\| \temp_n\| ^{1/2}\| \nabla \temp_n\| ^{3/2}.
  \end{equation}
  Applying Young's inequality, we obtain:
  \begin{equation}
    \label{eq:P1-Nirenberg-Young}
    c\| \temp_n\| ^{1/2}\| \nabla \temp_n\| ^{3/2} \le \eta \| \nabla \temp_n\| ^2+c_\eta \| \temp_n\| ^2.
  \end{equation}
  Finally, we obtain the structure:
  \begin{align*}
    &\half\partial _t\| \temp_n\| ^2
      +({\kappa_0} -2N\eta )\| \nabla \temp_n\| ^2
      + { \varepsilon  \cflux_0\|  \temp_n\| ^2_\ltr }
      \le
      c_\eta ^\delta\sumN\| \bcoll_i\| ^2\| \temp_n\| ^2.
  \end{align*}
  For a small $\eta > 0$ Gronwall's lemma gives:
  \begin{align*}
    &\| \temp_n(t)\| ^2+ {\kappa_0}  \int_0^t\| \nabla \temp_n(t)\| ^2<C&&\text{for }t\in (0,T),
  \end{align*}
  where $C>0$ is independent of $n$ { and $\varepsilon$}, since $\bcoll_i$ are uniformly bounded. This ensures that
  \begin{equation}
    \label{eq:P1-energy-estimate-1}
    {\{ \temp_n \}  \text{ is bounded in } \lilta \cap  \ltha.}
  \end{equation}

  To show uniform estimates for $\partial _t\temp_n$ with respect to $n$,
  we can take $\phi=\partial _t\temp_n$ in (\ref{eq:P1-weak-finite}).
  {Indeed, by the formula (\ref{eq:P1-Galerkin-approximation}) of
    $\theta_n^{\varepsilon}$,
    $\partial_t \theta_n^{\varepsilon} = \sum_{j= 1}^n (\partial_t \alpha_j^n) \xi_j$
    so that $\partial_t \theta_n^{\varepsilon} \in \text{span}\{\xi_j\}_{j=1}^n$.
    Then by using
    the Cauchy-Schwarz and Young's inequalities,
    as well as the mollifier property
    (\ref{eq:thermal.mollified-gradient-property}) we get: }
  \begin{align}
    &\| \partial _t\temp_n\| ^2
      +\half\partial _t\| \sqrt{\cond} \nabla \temp_n\| ^2
      + \varepsilon \frac{\cflux_0}{2}\partial _t\| \temp_n\| ^2_\ltr
      \le
      \tau_* \sumN\inta |\nabla ^\delta \bcoll_i \cdot  \nabla \temp_n\partial _t\temp_n|
      \nonumber\\
    \label{eq:P1-rev-20}
    &\quad
      \le
      \left(
      c^\delta \tau_* \sumN\| \bcoll_i\| _\lia
      \right)
      (\eta\| \partial _t\temp_n\| ^2+ C_\eta \| \nabla\temp_n\| ^2)
      \text{ for } \eta > 0.
  \end{align}
  { By taking a small $\eta > 0$ and  using (\ref{eq:P1-energy-estimate-1}), it holds that:
  }
  \begin{align*}
    & {\kappa_0} \| \nabla \temp_n\| ^2 + \int_0^t\| \partial _t\temp_n\| ^2<C&&\forAll{Time},
  \end{align*}
  where $C>0$ depends on $\delta$, but is independent of $n$ and $\varepsilon$.
  Together with (\ref{eq:P1-energy-estimate-1}) this ensures that:
  \begin{equation}
    \label{eq:P1-energy-estimate-2}
    { \{\temp_n \}  \text{ is bounded  in } \hlta \cap  \liha. }
  \end{equation}
  Hence, we can choose a subsequence $\temp_{n_k}\weakTo \temp$ in $\hlta$
  and $\temp_{n_k} \weakStarTo \temp$ in $\liha$ as $k\to\infty$.

  Now, using
  \begin{equation}
    \label{eq:P1-test-function}
    v_m(t,x):=\sum_{j=1}^m\beta_j^m(t)\xi_j(x)
  \end{equation}
  as a test function in (\ref{eq:P1-weak-finite}) and integrating
  with respect to time we get:
  \begin{equation}
    \label{P1-weak-finite-integrated}
    \begin{aligned}
      &\int[TA] \partial _t\temp_{n_k}v_m
      +
      \int[TA] \cond \nabla \temp_{n_k}\cdot \nabla v_m
      + {\varepsilon}
      \cflux_0\intt\intr \temp_{n_k}v_m
      \\&\qquad=
      \sumN\int[TA] \soret  \nabla ^\delta \bcoll_i\cdot \nabla \temp_{n_k}v_m.
    \end{aligned}
  \end{equation}
  Using (\ref{eq:P1-energy-estimate-2}), we pass to the limit as $k\to\infty$
  to obtain: For each $m$
  \begin{equation}
    \label{eq:P1-weak-integrated}
    \int[TA] \partial _t\temp v_m
    +
    \int[TA] \cond \nabla \temp\cdot \nabla v_m
    + {\varepsilon}
    \cflux_0\intt\intr \temp v
    =
    \sumN\int[TA]   \soret \nabla ^\delta \bcoll_i\cdot \nabla \temp v_m.
  \end{equation}
  Note that (\ref{eq:P1-weak-integrated}) holds for all $v\in \ltha$ since
  we can approximate $v$ with $v_m$ in $\ltha$, hence
  \begin{align*}
    \int[TA] \partial _t\temp v
    +
    \int[TA] \cond \nabla \temp\cdot \nabla v
    + {\varepsilon}
    \cflux_0\intt\intr \temp v
    =
    \sumN\int[TA] \soret  \nabla ^\delta \bcoll_i\cdot \nabla \temp v,
  \end{align*}
  holds for all $v\in \ltha$.

  Finally, we show the initial condition holds. Indeed, the Aubin-Lions lemma guarantees
  that $\temp_{n_i}\to\temp$ in $C([0,T]; \lta)$. Then on account of $\temp_{n_k}(0)\to\temp^0$
  in $\lta$ as $k\to\infty$, we get $\temp(0)=\temp^0$.
\end{proof}
\begin{lemma}{{\bf Positivity and boundedness of solutions to \nameref{auxiliary:1}.}}
  \label{lemma:P1-positivity-boundedness}\\
  \INPUT{
    Let $\bcoll_i\in K(T,M)$, $M>0$, and assume \nameref{assumption:1}-\nameref{assumption:2}.
  }
  Then $0\le \temp\le \| \temp^0\| _\lia$ a.e. $\tdomain{temp}$.
\end{lemma}
\begin{proof}
  Let $\temp:=\temp^+-\temp^-$,
  where $z^+:=\max(z,0)$ and $z^-:=\max(-z,0)$.
  Testing (\ref{eq:P1-weak}) with $\phi:=-\temp^-$, and using (\ref{eq:thermal.mollified-gradient-property}) gives:
  \begin{align*}
    &\half\partial _t\| \temp^-\| ^2+
      {\kappa_0} \| \nabla \temp^-\| ^2
      + {\varepsilon}  \cflux_0\| \temp^-\| ^2_\ltr
      \le
      c^\delta\soret\sumN\| \bcoll_i\| _\infty \| \nabla \temp^-\temp^-\| _\loa \\
    &\quad \le
      \left(
      C^\delta_\eta \soret\sumN\| \bcoll_i\| _\infty
      \right)
      \| \temp^-\| ^2+ {\eta}  \| \nabla \temp^-\| ^2 \text{ for } \eta > 0.
  \end{align*}
  Choosing { $\eta < \kappa_0$} and taking into account that $\temp^-(0)= 0$,
  Gronwall's lemma gives $\| \temp^-\| ^2\le 0$. This means $\temp \geq 0$ a.e. in $\Omega$ for
  all $t\in(0,T)$.

  Let $\phi = (\temp-M_0)^+$ in (\ref{eq:P1-weak}) with $M_0\ge \| \temp(0)\| _\lia$:
  For $\eta > 0$
  \begin{align*}
    &\half\partial _t \| (\temp-M_0)^+\| ^2+ \kappa_0 \| \nabla (\temp-M_0)^+\| ^2
      + {\varepsilon} \cflux_0\| (\temp-M_0)^+\| ^2_\ltr\\
    &\qquad+ g_0 \intr M_0(\temp-M_0)^+
      \le \tau_* \sumN\inta \nabla^\delta\bcoll_i\cdot \nabla(\temp-M_0)^+(\temp-M_0)^+\\
    &\qquad\le
      \left(
      \tau_* c^\delta\sumN\| \bcoll_i\| _\infty
      \right)
      \left(
      c_\eta \| (\temp-M_0)^+\| ^2+ {\eta}  \| \nabla(\temp-M_0)^+\| ^2
      \right).
  \end{align*}
  Discarding the positive terms on the left side and then applying Gronwall's lemma
  leads to:
  \begin{align*}
    \| (\temp-M_0)^+(t)\| ^2\le \| (\temp-M_0)^+(0)\| ^2
    \exp
    \left(
    \tau_* c^\delta c_\eta \sumN\| \bcoll_i\| _\infty t
    \right).
  \end{align*}
  Since $\| (\temp-M_0)^+(0)\| =0$, we obtain $(\temp-M_0)^+(t)=0$.
  Thus the proof of the lemma is completed.
\end{proof}
\begin{lemma}{\bf Existence of solutions to \nameref{auxiliary:2}.}
  \label{lemma:P2-existence}\\
  \INPUT{
    Let $\btemp\in K(T,M), M>0$
    and \nameref{assumption:1}-\nameref{assumption:2} hold.
  }
  Then \nameref{auxiliary:2} has solutions
  $\coll_i\in \hlta\cap L^\infty(0,T;H^1(\Omega))$
  and $\dep_i\in  \hltg$
  in the following sense:
  \begin{align}
    \intertext{For all $\psi_i\in \ha$, it holds:}
    \label{eq:P2-coll-weak}
    &\begin{aligned}
      \eqColloidsWeak{\coll_i}{\btemp}{R_i^M(\coll)}{\psi_i}
    \end{aligned}\\
    \label{eq:P2-coll-initial}
    &\coll_i(0,x)=\coll_i^0(x)\quad \text{a.e. in $\A$,}\\
    \intertext{and for all $\varphi_i\in \ltg$:}
    \label{eq:P2-dep-weak}
    & \intg \partial _t\dep_i \varphi_i= \intg (a_i\coll_i- b_i\dep_i)\varphi_i,\\
    \label{eq:P2-dep-initial}
    & \dep_i(0,x)=\dep_i^0(x)\quad\text{a.e. on $\G$}.
  \end{align}
\end{lemma}
\begin{proof}
  Let \{$\xi_j$\} -- Schauder basis of $\ha$.
  Then, for each $n\in\naturalNumbers$, there exists
  \begin{equation}
    \label{eq:P2-galerkin-initial}
    \coll_{i,n}^0(x):=\sum_{j=1}^n\alpha_{i,j}^{0,n}\xi_j(x)
    \text{ such that }
    \coll_{i,n}^0\to\coll_i^0\text{ in }\ha
    \text{ as } n\to\infty.
  \end{equation}
  We denote by $\coll_{i,n}$ the Galerkin approximation of $\coll_i$, that is:
  \begin{align}
    \label{eq:galarkin-def}
    \coll_{i,n}(t,x):=\sum_{j=1}^n\alpha_{i,j}^n(t)\xi_j(x) && \forAll{TA}.
  \end{align}
  $\coll_{i,n}$ must satisfy (\ref{eq:P2-coll-weak}), and hence,
  \begin{equation}
    \label{eq:P2-coll-weak-finite}
    \begin{aligned}
      \eqColloidsWeak{\coll_{i,n}}{\btemp}{R_i^M(\coll_n)}{\psi_i},
      \qquad\text{for all }\psi_i\in \mspan\{\xi_j\}_{j=1}^n.
    \end{aligned}
  \end{equation}
  Accordingly, let $\{\eta_j\}$ -- an orthonormal basis of $\ltg$.
  Then for each $n\in\naturalNumbers$ there exists
  \begin{equation}
    \label{eq:P2-galerkin-initial-dep}
    \dep_{i,n}^0(x):=\sum_{j=1}^n\beta_{i,j}^{0,n}\eta_j(x)
    \text{ such that }
    \dep_{i,n}^0\to\dep_i^0\text{ in }\ltg
    \text{ as } n\to\infty.
  \end{equation}
  We denote by $\dep_{i,n}$ the Galerkin approximation of $\dep_i$, that is:
  \begin{align}
    \label{eq:galarkin-def-dep}
    \dep_{i,n}(t,x):=\sum_{j=1}^n\beta_{i,j}^n(t)\eta_j(x) ,&& \forAll{TG}.
  \end{align}
  $\dep_{i,n}$ must satisfy (\ref{eq:P2-dep-weak}), and hence,
  \begin{equation}
    \label{eq:P2-dep-weak-finite}
    \eqDepositWeak{\dep_{i,n}}{\coll_{i,n}}{\varphi_i},
    \qquad\text{for all }\varphi_i\in \mspan\{\eta_j\}_{j=1}^n.
  \end{equation}
  $\alpha_{i,j}^n(t)$ and $\beta_{i,j}^n(t)$ can be found by substituting $\coll_{i,n}$ and $\dep_{i,n}$
  into (\ref{eq:P2-coll-weak}) -- (\ref{eq:P2-dep-initial})
  and using $\xi_k$ and $\eta_k$ for $k\inRange{n}$ as test functions:
  \begin{align}
    \label{eq:P2-ode}
    &\begin{aligned}
      &\partial _t\alpha_{i,k}^n(t)+
      \sum_{j=1}^n(A_{ijk}+ C_{ijk}-D_{ijk})\alpha_{i,j}^n(t)
      -\sum_{j=1}^nE_{ijk}\beta_{i,j}^n(t)
      \\&\qquad
      =
      \inta \xi_k
      \sum_{a=1}^{i-1}\coag_{a,i-a}
      \sigma_M\left(\sum_{b=1}^n\alpha_{a,b}^n(t)\xi_b\right)
      \sigma_M\left(\sum_{c=1}^n\alpha_{i-a,c}^n(t)\xi_c\right)
      \\&\qquad
      -
      \inta \xi_k
      \sum_{a=1}^N\coag_{a,i}
      \sigma_M\left(\sum_{b=1}^n\alpha_{i,b}^n(t)\xi_b\right)
      \sigma_M\left(\sum_{c=1}^n\alpha_{a,c}^n(t)\xi_c\right),
    \end{aligned}\\
    \label{eq:P2-ode-ic}
    &\alpha_{i,j}^n(0)=\alpha_{i,j}^{0,n},\\
    \label{eq:P2-ode-dep}
    &\begin{aligned}
      &\partial _t\beta_{i,k}^n(t)=\sum_{j=1}^nG_{ijk}\alpha_{i,j}^n(t)-H_{ijk}\beta_{i,j}^n(t),
    \end{aligned}\\
    \label{eq:P2-ode-ic-dep}
    &\beta_{i,j}^n(0)=\beta_{i,j}^{0,n}.
  \end{align}
  The coefficients arising in (\ref{eq:P2-ode}) are defined by:
  \begin{align*}
    &A_{ijk}:=\inta\diff_i\nabla \xi_j\cdot \nabla \xi_k,
    &&\\
    &C_{ijk}:= {\varepsilon a_i}\intg \xi_j\xi_k,
    &&D_{ijk}:=\inta \dufour_i \nabla ^\delta \btemp\cdot \nabla \xi_j \xi_k,\\
    &E_{ijk}:= {\varepsilon   b_i}\intg\xi_k\eta_j,
    &&G_{ijk}:= a_i\intg \xi_j\eta_k,\\
    &H_{ijk}:= b_i\intg \eta_j\eta_k.
  \end{align*}
  The left-hand side of this system of ODEs is linear, while the right-hand side
  is globally Lipschitz.
  Thus there exists a unique solution $\alpha_{i,j}^n(t), \beta_{i,j}^n(t)\in H^1(0,T)$
  to (\ref{eq:P2-ode}) - (\ref{eq:P2-ode-ic-dep}) for $t\in(0,T)$.

  To show uniform  { estimates in $n$ } for $\coll_{i,n}$ and $\dep_{i,n}$,
  we take $\psi_i=\coll_{i,n}$ and $\varphi_i=\dep_{i,n}$
  in (\ref{eq:P2-coll-weak-finite}) and (\ref{eq:P2-dep-weak-finite}) respectively.
  We get the inequality:
  \begin{align*}
    &\half\partial _t\| \coll_{i,n}\| ^2
      + {d_0}    \| \nabla \coll_{i,n}\| ^2
      +  {\varepsilon}  a_0 \| \coll_{i,n}\| ^2_\ltg\\
    &\quad\le
      {\varepsilon} b_* \intg |\dep_{i,n}\coll_{i,n}|
      + \delta_*  c^\delta\| \btemp\| _\infty\| \nabla \coll_{i,n}\| \| \coll_{i,n}\|
      +\int[A]R_i^M(\coll_n)\coll_{i,n}\\
    &\quad\le
      {\eta}  \| \coll_{i,n}\| ^2_\ltg
      + C^{{\eta}}  \| \dep_{i,n}\| ^2_\ltg
      + {\eta}  \| \nabla \coll_{i,n}\| ^2\\
    &\qquad
      +C^{\delta {\eta}  }\| \btemp\| _\infty\| u_{i,n}\| ^2+C^M\|u_{i,n}\|, \\
    &\half\partial _t\| \dep_{i,n}\| ^2_\ltg
      + \dcb_i\| \dep_{i,n}\| ^2_\ltg\\
    \le& {\eta}  \| \coll_{i,n}\| ^2_\ltg
         +C^{{\eta}}  \| \dep_{i,n}\| ^2_\ltg \\
    \le & { C \eta ( \| \nabla u_{i,n}^\varepsilon\|^2 +
       \| u_{i,n}^\varepsilon\|^2) +  C^{\eta} \| \dep_{i,n}\| ^2_\ltg \text{ for } \eta > 0.  }
  \end{align*}
  After {taking a small $\eta$ and }
  adding the two inequalities, Gronwall's lemma gives:
  \begin{align}
    \label{eq:P2-Gronwalls-1}
    \| \coll_{i,n}\| ^2+ d_0 \int_0^t\| \nabla \coll_{i,n}\| ^2
    + \| \dep_{i,n}\| ^2_\ltg<C&&\forAll{Time},
  \end{align}
  where $C>0$ depends on $\delta$, $M$ and $T$, but is independent of $n$ and $\varepsilon $, which ensures:
  \begin{align}
    \label{eq:P2-energy-estimate-1}
    & { \{ \coll_{i,n} \} \text{ is bounded in  } }  \lilta \cap  \ltha,\\
    \label{eq:P2-energy-estimate-2}
    &{ \{ \dep_{i,n} \} \text{ is bounded in }  } \liltg.
  \end{align}
  To show uniform estimates for $\partial _t\coll_{i,n}$ and $\partial _t\dep_{i,n}$ with respect to $n$,
  we take $\psi_i=\partial _t\coll_{i,n}$ and $\varphi_i=\partial _t\dep_{i,n}$
  in (\ref{eq:P2-coll-weak-finite}) and (\ref{eq:P2-dep-weak-finite}) respectively,
  noticing that they are in { span$\{\xi_j\}_{j = 1}^n$.}
  We obtain:
  \begin{align}
    \label{eq:P2-energy-interm-1}
    &\begin{aligned}
      &\| \partial _t\coll_{i,n}\| ^2
      +\int[A]\frac{\diff_i}{2}\partial _t(\nabla \coll_{i,n})^2
      +\frac{\varepsilon a_i}{2}\partial _t\| \coll_{i,n}\| ^2_\ltg\\
      =  &      {\varepsilon \intg b_i  } \partial _t\coll_{i,n}\dep_{i,n}
      + \int[A] \dufour_i  \nabla ^\delta\btemp\cdot \nabla \coll_{i,n}\partial _t\coll_{i,n}
      +\int[A]R_i^M(\coll_n)\partial _t\coll_{i,n} \\
      =  &      {
        \varepsilon \partial_t \intg b_i   \coll_{i,n}\dep_{i,n}
        -  \varepsilon  \intg b_i   \coll_{i,n} \partial_t \dep_{i,n}
        +\int[A]  \dufour_i \nabla ^\delta\btemp\cdot \nabla \coll_{i,n}\partial _t\coll_{i,n}
        +\int[A]R_i^M(\coll_n)\partial _t\coll_{i,n},
      }
    \end{aligned}
    \\
    \label{eq:P2-energy-interm-2}
    &\| \partial _t\dep_{i,n}\| ^2_\ltg
      + { \frac{b_i}{2}\partial _t\|   \dep_{i,n}\| ^2_\ltg }
      = a_i\intg\coll_{i,n}\partial _t\dep_{i,n}.
  \end{align}
  { Adding them,
    and finally integrating the result over $(0,t)$,}  we get:
  {
    \begin{align*}
      &\int_0^t \| \partial _t\coll_{i,n}\| ^2
        + \int_0^t  \| \partial _t\dep_{i,n}\| ^2 \\
      &\quad       +\frac{d_0}{2} \| \nabla \coll_{i,n}(t)\|^2  + \frac{ \varepsilon a_0}{2} \| \coll_{i,n}(t)\| ^2_\ltg
        +\frac{b_i}{2}\| \dep_{i,n}(t)\| ^2_\ltg  \\
      & \quad \le b_* \| \coll_{i,n}(t)\| _\ltg\| \dep_{i,n}(t)\| _\ltg
        + b_* \| \coll_{i,n}(0)\| _\ltg\| \dep_{i,n}(0)\| _\ltg\\
      &\quad +  \eta  \int_0^t  \| \partial _t\dep_{i,n}\| ^2
        + \varepsilon^2 c^{\eta} b_*^2 \int_0^t \| u_{i,n}^{\varepsilon}\|^2
        +    \frac{d_*}{2}\| \nabla \coll_{i,n}(0)\| ^2 \\
      &\quad             +\frac{\varepsilon a_*}{2}\| \coll_{i,n}(0)\| ^2_\ltg
        +\frac{b_*}{2}\| \dep_{i,n}(0)\| ^2_\ltg \\
      & \quad     + \eta \int_0^t \| \partial _t\coll_{i,n}\| ^2
        +  \delta_*^2 c^\delta c^{\eta}  \| \btemp\| _\infty\int_0^t \| \nabla \coll_{i,n}\| ^2 \\
      & \quad    +C^{M} C^{\eta} + \eta  a_* \int_0^t  \| \partial _t\coll_{i,n}\| ^2
        \text{ for } t \in (0,T] \text{  and } \eta > 0.
    \end{align*} }
  {Denoting the initial condition terms on the right as $C_0$ and using  (\ref{eq:P2-energy-estimate-1}) and (\ref{eq:P2-energy-estimate-2}), } we get: {
    \begin{align}
      &\int_0^t (1-2\eta )\| \partial _t\coll_{i,n}\| ^2
        + \int_0^t (1 - \eta)  \| \partial _t\dep_{i,n}\| ^2
        +\frac{d_0}{2}\| \nabla \coll_{i,n}(t)\| ^2\nonumber\\
      \label{eq:P2-revision-19}
      &\quad\le
        C_0+ a_* \delta_* c^\delta c^\varepsilon \| \btemp\| _\infty\intt\| \nabla \coll_{i,n}\| ^2+C^{M}C^\varepsilon \quad \text{ for } t \in (0,T].
    \end{align} }
  Then  { by using  (\ref{eq:P2-energy-estimate-1}),  again, we have:}
  \begin{equation*}
    \| \nabla \coll_{i,n}(t)\| ^2+\intt\| \partial _t\coll_{i,n}\| ^2+\intt\| \partial _t\dep_{i,n}\| ^2\le C \quad
    \mbox{ for } t \in (0,T],
  \end{equation*}
  where $C>0$ depends on $\delta$, $M$ and $T$, but is independent of $n$ and {$\varepsilon$.
    Namely,  this gives:}
  \begin{align}
    \label{eq:P2-derivative-energy-estimate-1}
    & { \{ \coll_{i,n} \} \text{ is bounded in }  \hlta \cap  \liha, } \\
    \label{eq:P2-derivative-energy-estimate-2}
    & { \{ \dep_{i,n} \} \text{ is bounded in }  \hltg. }
  \end{align}
  Hence, we can choose subsequences $\coll_{i,n_j} \weakTo \coll_i$ in $\hlta$
  and $\coll_{i,n_j}\to\coll_i$ in $C([0,T],\lta)$ and weakly$^*$ in  $\liha$   and
  $\dep_{i,n_j}\weakTo\dep_i$ in $\hltg$
  as $j\to\infty$. Since $R_i^M$ is Lipschitz continuous,
  the rest of the proof follows the same line of arguments as in Lemma~\ref{lemma:P1-existence}.
\end{proof}
\begin{lemma}{{\bf Positivity and boundedness of solutions to }\nameref{auxiliary:2}.}
  \label{lemma:P2-positivity-boundedness}\\
  Let $\btemp\in K(T,M)$, $M>0$ and assume \nameref{assumption:1}-\nameref{assumption:2}.
  Then $0\le \coll_i\le M_i(T+1)$ a.e. $\tdomain{coll}$,
  $0\le \dep_i\le \bar{M}_i(T+1)$ a.e. $\tdomain{dep}$,
  where $M_i>0$ and $\bar{M}_i>0$ are independent of $M$.
\end{lemma}
\begin{proof}
  {
    Testing (\ref{eq:P2-coll-weak}) with $\psi_i=-\coll_i^-$  and the definition of $R_i^M$ give:
    \begin{align*}
      &\half \partial _t\| \coll_i^-\| ^2
        + d_0 \| \nabla \coll_i^-\| ^2
        +\cflux_i\| \coll_i^-\| ^2_\ltr
        + \varepsilon a_0 \| \coll_i^-\| ^2_\ltg
        + \varepsilon \intg b_i \dep_i\coll_i^-
      \\
      &\qquad\le \delta_* c^\delta\| \btemp\| _\infty\int_{\Omega} |\nabla \coll_i^-\coll_i^-|
        - \inta \sum_{j=1}^{i-1}\coag_{j,i-j} \coll_j^+\coll_{i-j}^+\coll_i^-\\
      &\qquad\quad
        +\inta \sum_{j=1}^N\coag_{ij}\coll_i^+\coll_j^+\coll_i^-.
    \end{align*}
    The second term on the right is always negative, while the third is always zero.
    We can discard them and apply Cauchy-Schwarz and Young's inequalities
    to the first term on the right, as well as discard the positive terms on the left to obtain:
    \begin{align}
      \label{eq:lemma4-interm-1}
      \half\partial _t\| \coll_i^-\| ^2
      +(\diff_0 - \eta) \| \nabla \coll_i^-\| ^2
      \le
      \delta_* c^\delta c^{\eta}  \| \btemp\| _\infty \| \coll_i^-\| ^2+ b_* \intg \dep_i^-\coll_i^-
      \text{ for } \eta > 0.
    \end{align}
    Testing (\ref{eq:P2-dep-weak}) with $\varphi_i=-\dep_i^-$ gives:
    \begin{align}
      \label{eq:lemma4-interm-2}
      \half\partial _t\| \dep_i^-\| ^2_\ltg \le b_* \| \dep_i^-\| ^2_\ltg
      + a_* \intg \dep_i^-\coll_i^-.
    \end{align}
    We rely on Cauchy-Schwarz, Young's and trace inequalities to estimate the last term.
    We obtain:
    \begin{align*}
      \intg \dep_i^-\coll_i^-
      &\le \| \dep_i^-\| _\ltg\| \coll_i^-\| _\ltg\le c^\eta\| \dep_i^-\| _\ltg^2+\eta\| \coll_i^-\| _\ltg^2\\
      &\le c^\eta\| \dep_i^-\| _\ltg^2 + \eta C(\| \coll_i^-\| ^2+\| \nabla \coll_i^-\| ^2)
        \text{ for } \eta > 0.
    \end{align*}
    Adding (\ref{eq:lemma4-interm-1}) and (\ref{eq:lemma4-interm-2}) and choosing
    $\eta + \eta C <d_0$ and taking into account that $\coll_i^-(0)\equiv 0$ and
    $\dep_i^-(0) \equiv 0$, Gronwall's lemma
    gives $\| \coll_i^-\| ^2+\| \dep_i^-\| ^2\le 0$, that is $\coll_i\ge 0$ a.e. in $\A$
    and $\dep_i\ge 0$ a.e. in $\G$ for all $t \in (0,T]$.
  }

  {
    Next, let $i=1$ and $\psi_1 := (\coll_1-M_1)^+$ in (\ref{eq:P2-coll-weak})
    and $\varphi _1 := (\dep_1-\bar{M}_1)^+$ in (\ref{eq:P2-dep-weak}).
    Apply (\ref{eq:thermal.mollified-gradient-property}) for the cross-diffusion term to get:
    \begin{eqnarray*}
      & &\half\partial _t\| (u_1^{\varepsilon}-M_1)^+\| ^2
      + d_0 \| \nabla (u_1^{\varepsilon }-M_1)^+\| ^2
      + \varepsilon a_0 \| (u_1^{\varepsilon}-M_1)^+\| ^2_\ltg  \\
      & &   + \varepsilon \intg (a_1M_1- b_1\bar{M}_1) (u_1^{\varepsilon}-M_1)^+
      +\varepsilon  \intg b_1 (\dep_1-\bar{M}_1)^-(\coll_1-M_1)^+ \\
      &   \le &  \varepsilon \intg b_1 (\dep_1-\bar{M}_1)^+(u_1^{\varepsilon}-M_1)^+
      + \delta_* c^\delta\| \btemp\| _\infty
      \| \nabla (u_1^{\varepsilon}-M_1)^+ (u_1^{\varepsilon}-M_1)^+\| _\loa \\
      && + \int_{\Omega^{\varepsilon}} R_1^M(u^{\varepsilon}) (u_1^{\varepsilon}-M_1)^+,  \\
      &  &\half\partial _t\| (\dep_1-\bar{M}_1)^+\| ^2_\ltg
      + b_1 \| (\dep_1-\bar{M}_1)^+\| ^2_\ltg
      +\intg a_1(\coll_1-M_1)^-(\dep_1-\bar{M}_1)^+\\
      &  \le  &  \intg a_1 (\dep_1-\bar{M}_1)^+(\coll_1-M_1)^+
      +\intg (a_1M_1- b_1\bar{M}_1) (\dep_1-\bar{M}_1)^+.
    \end{eqnarray*}
    Here, by the definition we note that $R_1^M(u^{\varepsilon}) \le 0$.
    Also,  we choose $M_1$ and $\bar{M}_1$ such that $a_1M_1- b_1\bar{M}_1=0$ and
    add the two inequalities, while dropping the positive terms on the left and using Cauchy-Schwarz and Young's inequalities on the right to obtain:
    \begin{eqnarray*}
      & &  \half \partial _t\| (u_1-M_1)^+\| ^2
      + (d_0 -\eta )\| \nabla (u_1-M_1)^+\| ^2
      + \varepsilon  a_0 \| (u_1-M_1)^+\| ^2_\ltg\\
      & & + \half\partial _t\| (\dep_1-\bar{M}_1)^+\| ^2_\ltg \\
      & \le &  (a_* + \varepsilon  b_*)(\eta \| (\coll_1-M_1)^+\| ^2_\ltg  +c^\eta \| (\dep_1-\bar{M}_1)^
      +\| _\ltg) \\
      & & + c^\eta (\delta_* c^\delta)^2  \| \btemp\| _\infty^2  \| (u_1-M_1)^+\| ^2
      \text{ for } \eta > 0.
    \end{eqnarray*}
    Then by taking a small $\eta  > 0$  Gronwall's lemma gives:
    \begin{align*}
      &\| (\coll_1-M_1)^+(t)\| ^2+\| (\dep_1-\bar{M}_1)^+\| ^2_\ltg\\
      &\quad\le (\| (\coll_1-M_1)^+(0)\| ^2+\| (\dep_1-\bar{M}_1)^+(0)\| ^2_\ltg)
        \exp
        \left(
        C(\dufour_i,\btemp,\delta,M)t
        \right).
    \end{align*}
    Since we choose $M_1 > 0$ to satisfy $\| (\coll_1-M_1)^+(0)\| =0$,
    and $\bar{M}_1 > 0$ to satisfy $\| (\dep_1-\bar{M}_1)^+(0)\| _\ltg=0$, we get
    $0 \le \coll_1 \le M_1$ and $0\leq \dep_1 \le \bar{M}_1$.
  }

  {
    \newcommand\testcoll{(\coll_2-M_2(t+1))^+}
    \newcommand\testdep{(\dep_2-\bar{M}_2(t+1))^+}
    Let $i=2$ and $\psi_2 := \testcoll$ in (\ref{eq:P2-coll-weak})
    and $\varphi _2 := \testdep$ in (\ref{eq:P2-dep-weak}) with $a_2 M_2 = b_2 \bar{M}_2$:
    \begin{align*}
      \half\partial _t&(\| \testcoll\| ^2+\| \testdep\| ^2_\ltg)\\
              &+\frac{d_0}{2}\| \nabla \testcoll\| ^2\\
              &+ \varepsilon a_2 \| \testcoll\| ^2_\ltg+ b_2 \|  \testdep\| ^2_\ltg\\
              &\le C\| \testcoll\| ^2+\int[A] R_2^M(\coll)\testcoll\\
              &\quad-M_2\int[A]\testcoll
                -\bar{M}_2 \int_{\Gamma^{\varepsilon}} \testdep.
    \end{align*}
    Here, we note that
    \begin{equation*}
      R_2^M(\coll)\le \half\coag_{11}\sigma_M(\coll_1)^2\le \half\coag_{11}\coll_1^2\le \half\coag_{11}M_1^2.
    \end{equation*}
    Similarly, we have:
    \begin{align*}
      \half\partial _t&(\| \testcoll\| ^2+\| \testdep\| ^2_\ltg)\\
              &\le C\| \testcoll\| ^2+(\half\coag_{11}M_1^2-M_2)\int[A] \testcoll\\
              &\le C\| \testcoll\| ^2.
    \end{align*}
    By applying Gronwall's lemma with $\half\beta_{11}M_1^2\le M_2$,
    we see that $\coll_2\le M_2(T+1)$ $\indomain{TA}$
    and $\dep_2\le \bar{M}_2(T+1)$ $\indomain{TG}$. Recursively, we can obtain the same estimates
    for $\coll_i$ and $\dep_i$ for $i\ge 3$.
  }
\end{proof}

\begin{lemma}{{\bf The boundedness of the concentration gradient for} \nameref{auxiliary:2}.}
  \label{lemma:P1-grad-boundedness-coll}
  \\
  \INPUT{Let $\btemp\in K(T,M_0)$ and assume \nameref{assumption:1}-\nameref{assumption:2} to hold}.
  Then there exists a positive constant $C(M_0)$ such that $\| \nabla \coll_i(t)\| \le C(M_0)$ and $\int_0^T||\partial_t \coll_i(t)||^2dt\leq C(M_0)$ for $t\in \Time$.
\end{lemma}
\begin{proof}
  {
    Let $\coll_{i,n}$ be an approximate solution defined in the proof of Lemma~\ref{lemma:P2-existence}
    for each $n$. Then from (\ref{eq:P2-revision-19}) there exists a positive constant $C(M_0)$
    depending on $M_0$ such that
    \begin{align}
      \label{eq:P2-rev-19-2}
      \intt \| \partial _t\coll_{i,n}\| ^2\le C(M_0),&&\text{for each }n.
    \end{align}
    By letting $n\to\infty$ we have proved this Lemma. }
\end{proof}

\begin{lemma}{{\bf The boundedness of the temperature gradient for} \nameref{auxiliary:1}.}\\
  \label{lemma:P1-grad-boundedness}\\
  \INPUT{Let $\bcoll_i\in K(T,M_0)$ and assume \nameref{assumption:1}-\nameref{assumption:2} to hold}.
  Then there exists a positive constant $C(M_0)$ such that $\| \nabla \temp(t)\| \le C(M_0)$ and $\int_0^T\partial_t  \| \temp(t)||^2dt\leq C(M_0)$ for $t\in \Time$.
\end{lemma}
\begin{proof}
  From (\ref{eq:P1-rev-20}) we can prove this lemma in the similar way
  to that of Lemma~\ref{lemma:P1-grad-boundedness-coll}.
\end{proof}
\begin{theorem}{\bf Existence and uniqueness of weak solutions \model}\\
  \INPUT{Let \nameref{assumption:1}-\nameref{assumption:2} hold.}
  Then there exists a unique solution to \model.
\end{theorem}
\begin{proof}
  For any $M>0$, $X_M:=K(M,T)\times K(M,T)^N$ is a closed set of
  $X:= L^2(0,T; L^2(\Omega^{\varepsilon}))^{N+1}$.
  Let $\btemp_1, \btemp_2, \bcoll_{i,1}, \bcoll_{i,2}\in K(M,T)$,
  for $i\inRange{N}$, and put $\btemp:=\btemp_1-\btemp_2$, $\bcoll_{i}:=\bcoll_{i,1}-\bcoll_{i,2}$,
  {  $(\temp_1,\coll_{i,1})=\operatorT(\btemp_1,\bcoll_1)$ and
    $(\temp_2,\coll_{i,2})=\operatorT(\btemp_2,\bcoll_2)$,
    $\dep_{i,1} = \mathbf{T}_2((\btemp_1,\bcoll_1)$ and
    $\dep_{i,2} = \mathbf{T}_2((\btemp_2,\bcoll_2)$. }
  Moreover, we define
  $\temp=\temp_1-\temp_2$ and $\coll_i=\coll_{i,1}-\coll_{i,2}$ and $\dep_i=\dep_{i,1}-\dep_{i,2}$.

  By Lemma \ref{lemma:P1-positivity-boundedness} and Lemma \ref{lemma:P2-positivity-boundedness},
  $\operatorT: X_M\to X_M$ for $M>\max(\|\temp^0\|_\lia,M_1, M_2(T+1),\ldots,M_N(T+1))$.
  Hence, we want to prove the existence of a positive constant $C<1$ such that
  \begin{equation*}
    \|\operatorT(\btemp_1,\bcoll_{i,1})-\operatorT(\btemp_2,\bcoll_{i,2})\|_X\le
    C\|(\btemp_1,\bcoll_{i,1})-(\btemp_2,\bcoll_{i,2})\|_X
  \end{equation*}
  for small $T>0$.
  Substituting $\temp_1,\temp_2,\coll_{i,1},\coll_{i,2},\dep_1,\dep_2$ into the formulation:
  \begin{align*}
    &        \int[A] \partial _t\temp_1(\temp_1-\temp_2)
      +        \int[A] \cond \nabla \temp_1\nabla (\temp_1-\temp_2)
      + {\varepsilon} \cflux_0\int[ARobin] \temp_1(\temp_1-\temp_2)\\
    &\qquad
      =\sumN\int[A]  \soret  \nabla ^\delta\bcoll_{i,1}\cdot \nabla \temp_1(\temp_1-\temp_2),\\
    &        \int[A] \partial _t\temp_2(\temp_2-\temp_1)
      +        \int[A] \cond\nabla \temp_2\nabla (\temp_2-\temp_1)
      +  {\varepsilon} \cflux_0\int[ARobin] \temp_2(\temp_2-\temp_1)\\
    &\qquad
      =\sumN\int[A]  \soret \nabla ^\delta\bcoll_{i,2}\cdot \nabla \temp_2(\temp_2-\temp_1).
  \end{align*}
  Adding the last two equations we obtain:
  \begin{align*}
    &\half\partial _t\| \temp\| ^2+\cond^0\| \nabla \temp\| ^2+\cflux_0\| \temp\| ^2_\ltr
    \\&\qquad\le
        \tau_* \sumN
        \big|
        \underbrace{\inta(\nabla ^\delta\bcoll_{i,1}\cdot \nabla \temp_1-\nabla ^\delta\bcoll_{i,2}\cdot \nabla \temp_2)(\temp_1-\temp_2)}_A
        \big|.
  \end{align*}
  The term $A$ can be expressed as:
  \begin{align*}
    A&=\inta (\nabla ^\delta\bcoll_{i,1}\cdot \nabla \temp_1
       -\nabla ^\delta\bcoll_{i,2}\cdot \nabla \temp_1)(\temp_1-\temp_2)
    \\
     &+\inta (\nabla ^\delta\bcoll_{i,2}\cdot \nabla \temp_1-\nabla ^\delta\bcoll_{i,2}\cdot \nabla \temp_2)(\temp_1-\temp_2)\\
     &=
       \underbrace{\inta\nabla ^\delta\bcoll_i\cdot \nabla \temp_1\temp}_{A_1}
       +\underbrace{\inta\nabla ^\delta\bcoll_{i,2}\cdot \nabla \temp \temp}_{A_2}.
  \end{align*}
  With the help of Lemma~\ref{lemma:P1-grad-boundedness}, the terms $B$ and $C$ can be estimated
  as follows:
  \begin{align*}
    A_1&\le c^\delta M\| \bcoll_i\| ^2 + C(M)^2 \| \temp\| ^2, \\
    A_2 &\le c^\delta\| \bcoll_{i,2}\| _\infty (\eta\| \nabla\temp\| ^2+\frac{1}{4\eta}\| \temp\| ^2) \text{ for } \eta > 0.
  \end{align*}

  Looking at the formulation for the concentrations, we have:
  \begin{align*}
    &\int[A] \partial _t\coll_{i,1}(\coll_{i,1}-\coll_{i,2})
      +\int[A] \diff_i\nabla \coll_{i,1}\cdot \nabla (\coll_{i,1}-\coll_{i,2}) \\
    &\qquad
      +  {\varepsilon} a_i\intg\coll_{i,1}(\coll_{i,1}-\coll_{i,2})
      -  {\varepsilon} b_i\intg\dep_{i,1}(\coll_{i,1}-\coll_{i,2})\\
    &\qquad
      =\int[A] {\delta_i} \delta_i \nabla ^\delta\btemp_1\cdot\coll_{i,1}(\coll_{i,1}-\coll_{i,2})+
      \int[A] R_i(\coll_1)(\coll_{i,1}-\coll_{i,2}),\\
    &\int[A]\partial _t\coll_{i,2}(\coll_{i,2}-\coll_{i,1})
      +\int[A]\diff_i\nabla \coll_{i,2}\cdot \nabla (\coll_{i,2}-\coll_{i,1}) \\
    &\qquad
      +  {\varepsilon} a_i\intg\coll_{i,2}(\coll_{i,2}-\coll_{i,1})
      -  {\varepsilon} b_i\intg\dep_{i,2}(\coll_{i,2}-\coll_{i,1})\\
    &\qquad
      = \int[A] {\delta_i} \nabla ^\delta\btemp_2\cdot\coll_{i,2}(\coll_{i,2}-\coll_{i,1})+
      \int[A] R_i(\coll_2)(\coll_{i,2}-\coll_{i,1}).
  \end{align*}
  We also test the deposition equation with $\dep_i$ to obtain:
  \begin{align*}
    \half  {\partial_t} \|\dep_i\|_\ltg^2=
    { \int[G] a_i \dep_i \coll_i- b_i \| \dep_i\|_\ltg^2. }
  \end{align*}
  After adding the three above equations, we obtain:
  \begin{align*}
    &\half\partial _t\| \coll_i\| ^2
      +\half\partial _t\|\dep_i\|_\ltg^2
      +  {d_0}  \| \nabla \coll_i\| ^2
      +  {\varepsilon a_0} \| \coll_i\| ^2_\ltg\\
    &\qquad
      \le
      (a_* + {\varepsilon} b_*)\intg|\dep_i \coll_i|
      +\inta |(\nabla ^\delta\btemp_1\cdot \nabla \coll_{i,1}-\nabla ^\delta\btemp_2\cdot \nabla \coll_{i,2})\coll_i|
    \\&\qquad
        +\inta |(R_i(u_1)-R_i(u_2))u_i|,
  \end{align*}

  \begin{align*}
    &\half\partial _t\| \coll_i\| ^2
      +\half\partial _t\|\dep_i\|_\ltg^2
      +{d_0} \| \nabla \coll_i\| ^2
      +(a_0 -\eta)\| \coll_i\| ^2_\ltg
    \\
    \le   &\qquad
         \frac{({ a_*  +  \varepsilon b_*} )^2}{4\eta}\| \dep_i\| _\ltg^2
         +\underbrace{ \delta_* \inta |\nabla ^\delta\btemp_1\cdot \nabla \coll_i \coll_i|}_{B_1}\\
    &\qquad
      +\underbrace{ \delta_* \inta |\nabla \coll_{i,2}\cdot \nabla ^\delta\btemp \coll_i|}_{B_2}
      +\underbrace{\inta |(R_i(\coll_1)-R_i(\coll_2))\coll_i|}_{B_3},
  \end{align*}
  where the sub-expressions can be estimated as:
  \begin{align*}
    &B_1
      \le \eta\| \nabla \coll_i\| ^2+\frac{1}{4\eta} c^\delta \| \btemp_1\| ^2_\infty\| \coll_i\| ^2
      \text{ for } \eta > 0, \\
    &B_2 \le c^\delta C(M) \| \btemp\| ^2 + C(M) \| \coll_i\| ^2.
    \\\intertext{Note that with the boundedness of $\coll_i$ we can treat $R_i^M$ as a  Lipschitz continuous function with the Lipschitz constant $C_L$:}
    &B_3 \le C_L\| \coll_i\| ^2.
  \end{align*}
  Adding up the estimates for the temperature and concentrations:
  {
    \begin{align*}
      &  \frac{d}{dt} (\| \coll_i\| ^2 +  \| \dep_i\| ^2 + \| \temp\| ^2) +
        d_0 \| \nabla \coll_i\| ^2 + \kappa_0 \| \nabla \temp\| ^2 \\
      &\quad
        \le  c_1\| \coll_i\| ^2 + c_2\| \dep_i\| ^2 + c_3\| \temp\| ^2
        + c^\delta M (\| \bcoll_i\| ^2 + \| \btemp\| ^2).
    \end{align*} }
  Gronwall's lemma gives the estimate:
  \begin{align*}
    \| \temp(t)\| ^2+\| \coll_i(t)\| ^2\le
    C\left(
    \| \btemp\| ^2_\ltlta + \| \bcoll_i\| ^2_\ltlta
    \right).
  \end{align*}
  Integrating over $\Time$, we have:
  \begin{align*}
    \int[Time]\| \temp(t)\| ^2+\| \coll_i(t)\| ^2\le
    CT\left(
    \| \btemp\| ^2_\ltlta + \| \bcoll_i\| ^2_\ltlta
    \right).
  \end{align*}
  Accordingly, $\operatorT$ is a contraction mapping for $T'$ such that $CT'<1$.
  Then the Banach fixed point theorem shows that \model
  admits a unique solution in the sense of Definition \ref{def:weak-solution}
  on $[0,T']$. Next, we consider \model on $[T',T]$. Then we can solve uniquely this problem
  on $[T',2T']$. Recursively, we can construct a solution of \model on the whole
  interval $[0,T]$.
\end{proof}

\section{Passing to $\varepsilon\to 0$ (the homogenization limit)}\label{homogenization}

\subsection{Preliminaries on periodic homogenization}

Now that the well-posedness of our microscopic system is available, we can
investigate what happens as the parameter $\varepsilon$
vanishes. Recall that $\varepsilon$ defines both the microscopic
geometry and the periodicity in the model parameters.

\newcommand\var{u^{\varepsilon}}
\newcommand\varseq{(\var)}
\index{Two-scale convergence}
\begin{definition}
  \label{def:two-scale-convergence}
  (Two-scale convergence \cite{nguetseng1989general},\cite{allaire1992homogenization}).
  \renewcommand\A{\Omega}
  Let $\varseq$ be a sequence of functions in $\ltlta$, where
  $\A$ is an open set in $\Rdim{n}$ and $\varepsilon>0$ tends to $0$.
  $\varseq$ two-scale converges to a unique function $u_0(t,x,y)\in L^2((0,T)\times\Omega\times Y)$
  if and only if for all $\phi\in C_0^\infty((0,T)\times\Omega,C^\infty_{\#}(Y))$ we have:
  \begin{equation}
    \label{eq:two-scale-convergence}
    \lim_{\varepsilon\to 0}\int[Time]\inta \var\phi(t,x,\frac{x}{\varepsilon})dxdt=
    \frac{1}{|Y|}\int[Time]\inta\int\limits_{Y}u_0(t,x,y)\phi(t,x,y)dydxdt.
  \end{equation}
  We denote (\ref{eq:two-scale-convergence}) by $\var \twoScaleTo u_0$.
\end{definition}

The space $C^\infty_{\#}(Y)$ refers to the space of all $Y$-periodic  $C^\infty$-functions. The spaces $H^1_\#(Y)$ and $C^\infty_{\#}(\Gamma)$ have a similar meaning; the index $\#$ is always indicating that is about $Y$-periodic functions.

\begin{theorem} (Two-scale compactness on domains)
  \label{thm:two-scale-compactness}
  \renewcommand\A{\Omega}
  \begin{enumerate}[(i)]
  \item From each bounded sequence $\varseq$ in $\ltlta$, a subsequence may be extracted
    which two-scale converges to $u_0(t,x,y)\in L^2(\Time\times\A\times Y)$.
  \item Let $\varseq$ be a bounded sequence in $ L^2(0,T ; H^1(\Omega))$, then there exists
    $\tilde{u}\in L^2((0,T)\times \Omega;  H_\#^1(Y))$ such that up to a subsequence $\varseq$
    two-scale converges to $u_0\in \ltlta$ and $\nabla \var \twoScaleTo \nabla _xu_0 + \nabla _y\tilde{u}$.
  \end{enumerate}
\end{theorem}
\begin{proof}
  See e.g. \cite{nguetseng1989general},\cite{allaire1992homogenization}.
\end{proof}

\newcommand\ltwoGammaEps{L^2(\Time \times \Gamma_\varepsilon )}
\begin{definition}
  \label{def:two-scale-convergence-periodic}
  (Two-scale convergence for $\varepsilon $-periodic hypersurfaces \cite{neuss1996some}).
  A sequence of functions $\varseq\in \ltwoGammaEps$ is said to two-scale
  converge to a limit $u_0\in L^2(\Time \times\A\times \Gamma)$ if and only if
  for all $\phi\in C_0^\infty(\Time \times \A; C_\#^\infty(\Gamma))$ we have
  \begin{equation}
    \label{eq:two-scale-convergence-periodic}
    \lim_{\varepsilon \to 0}\varepsilon \int[Time]\int\limits_{\Gamma_\varepsilon} \var\phi(t,x,\frac{x}{\varepsilon})
    =\frac{1}{|Y|}\int[Time]\int\limits_\Omega\int\limits_\Gamma u_0(t,x,y)\phi(t,x,y) d\gamma_ydxdt.
  \end{equation}
\end{definition}

\begin{theorem} (Two-scale compactness on surfaces)
  \label{thm:two-scale-compactness-two}
  \begin{enumerate}[(i)]
  \item From each bounded sequence $\varseq\in \ltwoGammaEps$ one can extract a subsequence
    $\var$ which two-scale converges to $u_0\in L^2(\Time\times\Omega\times\Gamma)$.
  \item If a sequence $\varseq$ is bounded in $L^\infty(\Time\times \Gamma_\varepsilon)$, then
    $\var$  two-scale converges to a $u_0\in L^\infty(\Time\times\Omega\times\Gamma)$
  \end{enumerate}
\end{theorem}
\begin{proof}
  See \cite{neuss1996some} for proof of (i),
  and \cite{marciniak2008derivation} for proof of (ii).
\end{proof}

\begin{lemma}
  \label{convergence-lemma}
  \renewcommand\A{\Omega}
  \renewcommand\G{\Gamma}
  Let \allAssumptions hold. Denote by $\coll_i$ and $\temp$ the
  Bochner extensions\footnote{For our choice of microstructure, the
    interior extension from $H^1(\Omega^\varepsilon)$ into $H^1(\Omega)$
    exists and the corresponding extension constant is independent of the
    choice of $\varepsilon$; see the standard extension result reported in
    Lemma 5 from \cite{hornung1991diffusion}.} in the space
  $L^2(0,T;H^1(\Omega))$ of the corresponding functions originally
  belonging to $L^2(0,T;H^1(\Omega^\varepsilon))$. Then the following
  statement holds:
  \begin{enumerate}[(i)]
  \item \label{convergence-lemma-1}
    $\coll_i\weakTo\collBase_i$ and $\temp\weakTo\tempBase$ in $\ltha$,
  \item \label{convergence-lemma-2}
    $\coll_i\weakStarTo\collBase_i$ and $\temp\weakStarTo\tempBase$ in $\lilia$,
  \item \label{convergence-lemma-3}
    $\partial _t\coll_i\weakTo\partial _t\collBase_i$ and $\partial _t\temp\weakTo\partial _t\tempBase$ in $\ltlta$,
  \item \label{convergence-lemma-4}
    $\coll_i\to\collBase_i$ and $\temp\to\tempBase$
    strongly in $L^2(0,T;H^\beta(\A))$ for $\half <\beta <1$ and
    $\sqrt{\varepsilon }\| \coll_i-\collBase_i\| _{L^2(\Time\times\Gamma_{\varepsilon})}\to 0$ as $\varepsilon \to 0$,
  \item \label{convergence-lemma-5}
    $\coll_i\twoScaleTo \collBase_i$, $\nabla \coll_i\twoScaleTo\nabla _x\collBase_i+\nabla _y\collBase_i^1$
    where $\collBase_i^1\in L^2(\Time\times\A;H^1_\#(Y))$,
  \item \label{convergence-lemma-6}
    $\temp\twoScaleTo \tempBase$, $\nabla \temp\twoScaleTo\nabla _x\tempBase+\nabla _y\tempBase^1$
    where $\tempBase^1\in L^2(\Time\times\A;H^1_\#(Y))$,
  \item \label{convergence-lemma-7}
    $\dep_i\twoScaleTo\depBase_i\in L^\infty(\Time\times\A\times\G)$
    and
    $\partial _t\dep_i\twoScaleTo\partial _t\depBase_i\in L^2(\Time\times\A\times\G)$.
  \end{enumerate}
\end{lemma}
\begin{proof}
  \renewcommand\A{\Omega}
  We obtain (\ref{convergence-lemma-1}) and (\ref{convergence-lemma-2})
  as a direct consequence of the fact that $\coll_i$ and $\temp$
  are uniformly bounded in $\liha\cap \lilia$.
  A similar argument gives (\ref{convergence-lemma-3}).
  We get (\ref{convergence-lemma-4}) using the compact embedding
  $H^{\alpha}(\Omega)\hookrightarrow H^\beta(\Omega)$
  for $\beta\in (\half,1)$ and $0<\beta<\alpha\le 1$, since $\A$ has Lipschitz boundary.
  Note that (\ref{convergence-lemma-4}) implies the strong convergence
  of $\coll_i$ up to the boundary.

  Denote $W:=\{w\in \ltha\text{ and } \partial _tw\in \ltlta\}$. We have $\coll_i,\temp\in W$.
  Using Lions-Aubin lemma \cite{lions1969quelques} we see that $W$ is compactly embedded in
  $L^2(0,T;H^\beta(\A))$ for $\beta\in [0.5,1]$.
  We then use the trace inequality for perforated medium from \cite{hornung1991diffusion},
  namely for all $\phi\in H^1(\A^\varepsilon)$
  there exists a constant $C$ independent of $\varepsilon$ such that:
  \begin{equation}
    \label{eq:hornung-trace}
    \varepsilon \| \phi\| _{L^2(\G)}\le C(\| \phi\| ^2_{L^2(\A^\varepsilon )} + \varepsilon ^2\| \nabla \phi\| ^2_{L^2(\A^\varepsilon )}).
  \end{equation}
  Applying (\ref{eq:hornung-trace}) to $\coll_i-\collBase_i$, we get:
  \begin{align}
    \label{eq:eps-estimate}
    \sqrt{\varepsilon}\| \coll_i-\collBase_i\| ^2_{L^2(\Time\times \G)}
    &\le C\| \coll_i-\collBase_i\| ^2_{L^2(0,T;H^\beta(\A^\varepsilon ))}
      \nonumber\\
    &\le C\| \coll_i-\collBase_i\| ^2_{L^2(0,T;H^\beta(\A))},
  \end{align}
  where $\| \coll_i-\collBase_i\| ^2_{L^2(0,T;H^\beta(\A))}\to 0$ as $\varepsilon \to 0$.
  As for the rest of the statements (\ref{convergence-lemma-5})-(\ref{convergence-lemma-7}), since $\coll_i$ are bounded in $\liha$,
  up to a subsequence we have that $\coll_i\twoScaleTo \collBase_i$ in $\ltlta$,
  and $\nabla \coll_i\twoScaleTo \nabla _x u_i +\nabla _y\collBase_i^1$,
  where $\collBase_i^1\in L^2(\Time\times\A;H^1_\#(Y))$.
  By Theorem \ref{thm:two-scale-compactness-two},
  $\dep_i\twoScaleTo\depBase_i\in L^\infty(\Time\times\A\times\G)$
  and
  $\partial _t\dep_i\twoScaleTo\partial _t\depBase_i\in L^2(\Time\times\A\times\G)$.
\end{proof}
\subsection{Two-scale homogenization procedure}
\renewcommand\Diff{\mathbb{D}}
\newcommand\Soret{\mathbb{T}}
\newcommand\Cond{\mathbb{K}}
\newcommand\Dufour{\mathbb{F}}
\renewcommand\intY{\int\limits_{Y}}
\renewcommand\intG{\int\limits_{\Gamma}}
\newcommand\volY{|Y|}
\newcommand\volG{|\Gamma|}
\newcommand\DCA{A}
\newcommand\DCB{B}
\begin{theorem}
  \label{thm:two-scale-homogenization-procedure}
  \renewcommand\A{\Omega}
  Let \allAssumptions hold.  {
    The  limit functions $\theta$, $u_i$, $v_i$, $\theta^1$ and $u_i^1$ satisfy
    (\ref{two-scale-homogenization-procedure-1b}),
    (\ref{two-scale-homogenization-procedure-2b}) and
    (\ref{two-scale-homogenization-procedure-3b})  for any
    $\alpha\in C^\infty(\Time\times\Omega)$ and $\beta\in C^\infty(\Time\times\Omega;C^\infty_\#(Y))$.
  }
\end{theorem}
\begin{proof}
  \newcommand\yy{\frac{x}{\varepsilon}}
  Testing \model with oscillating functions $\phi(t,x)=\alpha(t,x)+\varepsilon \beta(t,x,\frac{x}{\varepsilon })$,
  where $\alpha\in C^\infty(\Time\times\Omega)$ and $\beta\in C^\infty(\Time\times\Omega;C^\infty_\#(Y))$,
  we obtain:
  \begin{align}
    \label{two-scale-homogenization-procedure-1a}
    & \inta \partial _t\temp (\alpha + \varepsilon \beta)
      + \inta \cond(\yy) \nabla \temp (\nabla _x\alpha+\varepsilon \nabla _x\beta+\nabla _y\beta)\nonumber\\
    &\qquad + \cflux_0\varepsilon \intg \temp (\alpha+\varepsilon \beta) =
      \sumN \inta \soret(\yy) \nabla ^\delta\coll_i\cdot \nabla \temp (\alpha+\varepsilon \beta),
    \\\label{two-scale-homogenization-procedure-2a}
    & \inta \partial _t\coll_i (\alpha + \varepsilon \beta)
      + \inta \diff_i(\yy) \nabla \coll_i (\nabla _x\alpha+\varepsilon \nabla _x\beta+\nabla _y\beta)\nonumber\\
    &\qquad  
      +\varepsilon \intg (a_i\coll_i- b_i\dep_i)(\alpha+\varepsilon \beta)\nonumber\\
    &\qquad =
      \inta\dufour_i( \frac{x}{\varepsilon} ) \nabla ^\delta\temp\cdot \nabla \coll_i (\alpha+\varepsilon \beta) + \inta R_i(\coll) (\alpha+\varepsilon \beta),
    \\\label{two-scale-homogenization-procedure-3a}
    & \varepsilon \intg \partial _t\dep_i (\alpha + \varepsilon \beta) =
      \varepsilon \intg ( a_i \coll_i- b_i \dep_i) (\alpha + \varepsilon \beta).
  \end{align}
  Using the concept of two-scale convergence for $\varepsilon \to 0$ in
  (\ref{two-scale-homogenization-procedure-1a}),
  (\ref{two-scale-homogenization-procedure-2a}) and
  (\ref{two-scale-homogenization-procedure-3a})
  yields:
  \ifdefined\included
  \renewcommand\A{\Omega}
  \renewcommand\intg{\int\limits_{\Gamma}}
  \renewcommand\cond{\kappa}
  \renewcommand\diff{d}
  \renewcommand\dca{a}
  \renewcommand\dcb{b}
  \renewcommand\dufour{\delta}
  \renewcommand\soret{\tau}
  \fi
  \begin{align}
    \label{two-scale-homogenization-procedure-1b}
    & \int_{\Omega} \partial _t\tempBase  \alpha
      + \frac{1}{ |Y_1| }\int_{\Omega}  \int_{Y_1}  \kappa(y)
      (\nabla \tempBase +\nabla _y\tempBase^1) (\nabla _x\alpha(x)+\nabla _y\beta(x,y))\nonumber\\
    &\qquad + \cflux_0\frac{|\Gamma_R|}{|Y_1|}\int_{\Omega} \tempBase \alpha =
      \sumN \frac{1}{|Y_1|} \int_{\Omega} \int_{Y_1}  \tau(y) \nabla ^\delta\collBase_i\cdot (\nabla \tempBase+\nabla _y\tempBase^1) \alpha,
    \\\label{two-scale-homogenization-procedure-2b}
    & \int_{\Omega}  \partial _t\collBase_i \alpha
      + \frac{1}{|Y_1|} \int_{\Omega} \int_{Y_1}  d_i(y)
      (\nabla \collBase_i + \nabla _y \collBase_i^1) (\nabla _x\alpha+\nabla _y\beta)\nonumber\\
    &\qquad  
      + \frac{1}{|Y_1|} \int_{\Omega} \intG (a_i \collBase_i- b_i \depBase_i) \alpha\nonumber\\
    &\qquad =
      \frac{1}{|Y_1| }\int_{\Omega} \int_{Y_1} \delta_i(y)
      \nabla ^\delta\tempBase\cdot (\nabla \collBase_i + \nabla _y\collBase_i^1)\alpha
      + \int_{\Omega} R_i(\collBase) \alpha,
    \\\label{two-scale-homogenization-procedure-3b}
    & \int_{\Omega}  \int_{\Gamma} \partial _t\depBase_i \alpha =
      \frac{1}{|Y_1|}\int_{\Omega} \intG (a_i \collBase_i - b_i \depBase_i) \alpha.
  \end{align}
  Note that we have used strong convergence for passing to the limit in the aggregation term in
  (\ref{two-scale-homogenization-procedure-2b}).
\end{proof}

Now we just need to find $\tempBase^1$ and $\collBase_i^1$.

\begin{lemma}
  The limit functions $\theta_1$ and  $u_i^1$ depend linearly on $\theta$ and $u_i$ as follows:
  \begin{eqnarray}
    &  &\tempBase^1:=\sum_{j=1}^3 \partial _{x_j} \theta \bar{\theta}^j,\\
    &  &\collBase_i^1:=\sum_{j=1}^3 \partial _{x_j} u_i \bar{u}_i^j.
  \end{eqnarray}
  Moreover,  $\bar{\theta}^j$ and $ \bar{u}_i^j$ solve the elliptic problems on the cell:
  (\ref{EP1}) and (\ref{EP2}),  respectively:
  \begin{equation}
    \left\{ \begin{array}{ll}
        \displaystyle
        -\nabla _y\cdot (\kappa(y) \nabla _y \bar{\theta}^j) =  \frac{\partial \kappa}{\partial y_j}
        & \text{ in } Y_1,\\
        \kappa \nabla_y  \bar{\theta}^j\cdot {\nu} = - \kappa\nu_j  & \text{ on } \Gamma, \\[0.2cm]
        \bar{\theta}^j \text{ is periodic in } Y, &
      \end{array} \right. \label{EP1}
  \end{equation}

  \begin{equation}
    \left\{ \begin{array}{ll}
        \displaystyle -\nabla _y \cdot (d_i(y) \nabla _y \bar{u}_i^j) =  \frac{\partial d_i}{\partial y_j}
        & \text{ in } Y_1,\\
        d_i \nabla_y  \bar{u}_i^j\cdot {\nu}  = - d_i \nu_j  &\text{ on } \Gamma, \\[0.2cm]
        \bar{\theta}^j \text{ is periodic in } Y, &
      \end{array} \right. \label{EP2}
  \end{equation}
\end{lemma}
\begin{proof}
  To do this we choose $\alpha=0$ in
  (\ref{two-scale-homogenization-procedure-1b}) and
  (\ref{two-scale-homogenization-procedure-2b}).
  This gives for all $\beta\in C^\infty(\Time\times\Omega;C^\infty_\#(Y))$
  a system of decoupled equations:
  \begin{align}
    \label{two-scale-homogenization-procedure-1c}
    & \int_{\Omega} \int_{Y_1}  \kappa(y) (\nabla \tempBase+\nabla _y\tempBase^1) \nabla _y\beta(x,y)=0,
    \\\label{two-scale-homogenization-procedure-2c}
    & \int_{\Omega} \int_{Y_1} d_i(y) (\nabla \collBase_i+\nabla _y\collBase_i^1) \nabla _y\beta(x,y)=0.
  \end{align}
  From these equations we can easily get the assertion of this lemma.
\end{proof}

\newcommand\tempCell{\bar{\tempBase}}
\newcommand\collCell{\bar{\collBase}}

\subsection{Strong formulation for the limit functions }\label{strong}
Here, we give the strong formulation $(P^0)$ for limit functions $\theta$, $u_i$ and $v_i$ obtained by  Lemma \ref{convergence-lemma}.
\begin{lemma}{\bf (Strong formulation).}
  \label{convergence-lemma-strong}
  Assume \nameref{assumption:1}-\nameref{assumption:2} to hold.  Then
  the triplet $(\theta, u_i, v_i)$ of limit functions of weak solutions to the microscopic model is a
  the weak solution of the  following macroscopic problem:
  \begin{eqnarray*}
    & & \partial_t \theta -  \nabla \cdot ( \mathbb{K}  \nabla \theta)
    + g_0 \frac{|\Gamma_R|}{|Y_1|} \theta
    = \sum_{i=1}^N  (\mathbb{T} \nabla^{\delta}u_i)  \cdot \nabla \theta
    \text{ in } (0,T) \times \Omega, \\
    & & - (\mathbb{K}  \nabla \theta)  \cdot {\nu}   = 0 \text{ on } (0,T) \times \partial \Omega,
  \end{eqnarray*}
  where $\mathbb{K}$ and $\mathbb{T}^i$ are  matrices given by $\mathbb{K} = K_0 \mathbb{I} + (K_{ij})_{ij}$ and
  $\mathbb{T} = T_0 \mathbb{I} + (T_{jk}^i)_{jk}$, respectively, $\mathbb{I}$ is the identity matrix,
  $$ K_0 = \frac{1}{|Y_1|}  \int_{Y_1} \kappa dy,  \quad
  {K}_{ij} = \frac{1}{|Y_1|}   \int_{Y_1} \kappa  \frac{\partial \bar{\theta}^j}{\partial y_i} dy, $$
  $$ {T}_0^i = \frac{1}{|Y_1|}  \int_{Y_1} \tau_i dy,  \quad
  {T}_{jk}^i = \frac{1}{|Y_1|} \int_{Y_1} \tau_i  \frac{\partial \bar{\theta}^j}{\partial y_k} dy,      $$
  and
  \begin{eqnarray*}
    & & \partial_t u_i -  \nabla \cdot ( \mathbb{D}^i \nabla u_i)
    + A_i u_i - B_i v_i   =   (\mathbb{F}^i \nabla u_i ) \cdot \nabla^{\delta} \theta
    + R_i(u)             \text{ in } (0,T) \times \Omega, \\
    & & - (\mathbb{D}^i \nabla u_i)  \cdot {\nu}  = 0 \text{ on } (0,T) \times \partial \Omega,
  \end{eqnarray*}
  where $\mathbb{D}^i$ and  $\mathbb{F}^i$ are matrices defined by $\mathbb{D}^i = D_i  \mathbb{I} + \mathbb{D}_0^i$ and
  $\mathbb{F}^i = F_i \mathbb{I} + \mathbb{F}_0^i$,
  $$ D_i = \frac{1}{|Y_1|}  \int_{Y_1} d_i dy,  \quad
  \mathbb{D}_0^i = (\frac{1}{|Y_1|}   \int_{Y_1} d_i  \partial_{y_k} \bar{u}_i^j dy)_{jk}, $$
  $$ F_i = \frac{1}{|Y_1|}  \int_{Y_1} \delta_i dy,  \quad
  \mathbb{F}^i = (\frac{1}{|Y_1|}  \int_{Y_1} \delta_i  \partial_{y_k} \bar{u}_i^j  dy)_{jk},   $$
  $$ A_i =  \frac{1}{|Y_1|}  \int_{\Gamma} a_i, \quad
  B_i =  \frac{1}{|Y_1|}  \int_{\Gamma} b_i, $$
  and initial conditions:
  \begin{align}
    & \tempBase(0,x)=\tempBase^0(x) &&\text{in }\Omega,\\
    & \collBase_i(0,x)=\collBase_i^0(x) &&\text{in }\Omega,\\
    & \depBase_i(0,x)=\depBase_i^0(x) &&\text{on }\Gamma.
  \end{align}
\end{lemma}
\begin{proof}
  First, choose $\alpha \in C^{\infty}((0,T) \times \Omega)$ and
  $\beta=0$ in (\ref{two-scale-homogenization-procedure-1b}) to obtain:
  \begin{align}
    \label{two-scale-homogenization-procedure-1f}
    & \int_{\Omega}  \partial _t\tempBase \alpha
      + \frac{1}{|Y_1|}\int_{\Omega} \int_{Y_1} \kappa(y)
      (\nabla \tempBase +
      \nabla _y(\sum_{j=1}^3 \partial _{x_j} \tempBase \tempCell^j)
      \nabla _x\alpha(x))\nonumber\\
    & + \cflux_0\frac{|\Gamma_R|}{|Y_1|} \int_{\Omega} \tempBase \alpha =
      \sumN \frac{1}{|Y_1|} \int_{\Omega} \int_{Y_1}
      \tau(y) \nabla ^\delta\collBase_i\cdot
      (\nabla \tempBase
      +      \nabla _y(\sum_{j=1}^3 \partial _{x_j} \tempBase \tempCell^j) \alpha.
  \end{align}
  Integrating (\ref{two-scale-homogenization-procedure-1f}) w.r.t. $y$ leads to:
  \begin{align}
    \label{two-scale-homogenization-procedure-1g}
    & \int_{\Omega}  \partial _t \tempBase \alpha
      +\int_{\Omega} \Cond \nabla \tempBase \nabla _x\alpha
      + \cflux_0\frac{1}{|Y_1|} \int_{\Omega} \int_{Y_1}   \tempBase  \alpha
      =
      \sumN \int_{\Omega}
      \Soret \nabla ^\delta\collBase_i \cdot
      \nabla \tempBase
      \alpha.
  \end{align}
  We can similarly derive from (\ref{two-scale-homogenization-procedure-2b}) that:
  \begin{align}
    \label{two-scale-homogenization-procedure-2g}
    & \int_{\Omega} \partial _t\collBase_i \alpha
      + \int_{\Omega} \mathbb{D}^i \nabla \collBase_i \nabla _x\alpha
      + \int_{\Omega}  (\DCA_i\collBase_i-\DCB_i\depBase_i)\alpha
      =
      \int_{\Omega} \Dufour^i \nabla ^\delta\tempBase\cdot \nabla \collBase_i\alpha
      + \int_{\Omega}  R_i(\collBase) \alpha,\\
    \label{two-scale-homogenization-procedure-3g}
    &\int_{\Omega} \partial _t\depBase_i \alpha =
      \int_{\Omega} (\DCA_i\collBase_i-\DCB_i\depBase_i)\alpha.
  \end{align}

  See also \cite{marciniak2008derivation} and \cite{fatima2012sulfate} for a similar application
  of the two-scale convergence method.
\end{proof}

\ifdefined\included\else
\section*{Acknowledgments}
AM and OK  gratefully acknowledge financial support by the European Union through
the Initial Training Network \emph{Fronts and Interfaces in Science and Technology}
of the Seventh Framework Programme (grant agreement number 238702).
\printbibliography
\end{document}